\documentclass[12pt]{amsart}
\usepackage{amsmath,amsfonts,amssymb,amsthm}
\usepackage{graphics}
\usepackage{epsfig}
\usepackage[usenames,dvipsnames]{color}
\usepackage{tikz}
\usepackage{psfrag}
\usepackage{mathtools}
\usepackage[colorlinks=true,linkcolor=black,menucolor=black,citecolor=black]{hyperref}
\usepackage[margin = 1in]{geometry}
\usepackage{cite}

\renewcommand{\d}{\ensuremath{\mathrm{d}}}
\newcommand{\R}{\ensuremath{\mathbb{R}}}

\newcommand{\N}{\ensuremath{\mathbb{N}}}

\newcommand{\TT}{\ensuremath{\mathcal T}}
\newcommand{\II}{\ensuremath{\mathcal I}}
\newcommand{\JJ}{\ensuremath{\mathcal J}}

\newcommand{\SSS}{\ensuremath{\mathcal S}}

\newcommand{\eps}{\ensuremath{\varepsilon}}
\newcommand{\verti}[1]{\ensuremath{\left\lvert #1 \right\rvert}}
\newcommand{\vertii}[1]{\ensuremath{\left\lVert #1 \right\rVert}}
\newcommand{\vertiii}[1]{{\left\lvert\kern-0.25ex\left\lvert\kern-0.25ex\left\lvert #1
    \right\rvert\kern-0.25ex\right\rvert\kern-0.25ex\right\rvert}}

\renewcommand{\d}{\ensuremath{{\rm d}}}

\newcommand{\ignore}[1]{}

\newtheorem{theorem}{Theorem}[section]

\newtheorem{proposition}[theorem]{Proposition}

\newtheorem{lemma}[theorem]{Lemma}
\newtheorem{corollary}[theorem]{Corollary}

\numberwithin{equation}{section}
\def\R{\mathbb{R}}

\def\N{\mathbb{N}}
\def\I{\infty}

\def\txtc{{\textnormal{c}}}
\def\txtd{{\textnormal{d}}}
\def\txte{{\textnormal{e}}}

\def\txts{{\textnormal{s}}}
\def\txtu{{\textnormal{u}}}
\def\txtD{{\textnormal{D}}}

\newcommand{\be}{\begin{equation}}
\newcommand{\ee}{\end{equation}}
\newcommand{\bea}{\begin{eqnarray}}
\newcommand{\eea}{\end{eqnarray}}
\newcommand{\beann}{\begin{eqnarray*}}
\newcommand{\eeann}{\end{eqnarray*}}
\newcommand{\benn}{\begin{equation*}}
\newcommand{\eenn}{\end{equation*}}

\def\ra{\rightarrow}
\def\I{\infty}

\newcommand{\cB}{{\mathcal B}}  
\newcommand{\cF}{{\mathcal F}}  
\newcommand{\cG}{{\mathcal G}}  
\newcommand{\cK}{{\mathcal K}}  
\newcommand{\cN}{{\mathcal N}}  
\newcommand{\cO}{{\mathcal O}}  
\newcommand{\cT}{{\mathcal T}}  

\begin{document}
%
\title[A dynamical systems approach for the contact-line singularity]{A dynamical 
systems approach for the contact-line singularity in thin-film flows}
\thanks{The authors are grateful for discussions with Lorenzo Giacomelli, Hans Kn\"upfer, and Felix Otto. MVG acknowledges financial support for precedent research provided by the International Max Planck Research School (IMPRS) of the Max Planck Institute for Mathematics in the Sciences (MIS) in Leipzig. MVG thanks the University of Monastir and the Vienna University of Technology for the kind hospitality. MVG was partially supported by Fields Institute for Research in Mathematical Sciences in Toronto and the National Science Foundation under Grant No.~NSF DMS-1054115. CK would like to thank the Austrian Academy of Science (\"OAW) for support via an APART Fellowship and the EU/REA for support via a Marie-Curie Integration Re-Integration Grant.}
\date{\today}
\author{Fethi Ben~Belgacem}
\address{UR Analyse Non-Lin\'eare et G\'eometrie (UR13ES32), Universit\'e de Monastir, Institut\linebreak
Sup\'erieur d'Informatique et de Math\'ematiques de Monastir, Avenue de la Corniche, B.P.~223, Monastir 5000, Tunisie}
\email{fethi.benbelgacem@fst.rnu.tn, belgacem.fethi@gmail.com}
\author{Manuel V. Gnann}
\address{University of Michigan, Department of Mathematics, 2074 East Hall, 530 Church Street, 
Ann Arbor, MI 48109-1043, United States}
\email{mvgnann@umich.edu}
\author{Christian Kuehn}
\address{Vienna University of Technology, Institute for Analysis and Scientific Computing, 
Wiedner Hauptstr. 8-10, 1040 Vienna, Austria}
\email{ck274@cornell.edu}
\begin{abstract}
We are interested in a complete characterization of the contact-line singularity of thin-film flows for zero and nonzero contact angles. By treating the model problem of source-type self-similar solutions, we demonstrate that this singularity can be understood by the study of invariant manifolds of a suitable dynamical system. In particular, we prove regularity results for singular expansions near the contact line for a wide class of mobility exponents and for zero and nonzero dynamic contact angles. Key points are the reduction to center manifolds and identifying resonance conditions at equilibrium points. The results are extended to radially-symmetric source-type solutions in higher dimensions. Furthermore, we give dynamical systems proofs for the existence and uniqueness of self-similar droplet solutions in the nonzero dynamic contact-angle case. 
\end{abstract}
\maketitle
\tableofcontents
%

\section{Introduction}
In this paper, we study the thin-film equation~\cite{Bertozzi}
\begin{subequations}\label{source_problem}
\begin{equation}\label{tfe}
\partial_t h + \partial_z \left(h^n \partial_z^3 h\right) = 0 \quad 
\mbox{for } \, (t,z) \in \{h > 0\}
\end{equation}
subject to the initial condition
\begin{equation}\label{initial}
\lim_{t \searrow 0} h = M \delta_0 \quad \mbox{in } \, \mathcal D^\prime(\R),
\end{equation}
\end{subequations}
where $M > 0$ is a constant (the mass of the droplet), $\mathcal{D}^\prime(\R)$ is the space of distributions, i.e., the dual space to the space of compactly supported test functions, and $\delta_0$ the Dirac distribution in $z = 0$. The function $h = h(t,z)$ describes the height of a viscous thin film as a function of time $t > 0$ and position $z \in \R$~\cite{beimr.2009,odb.1997}. The parameter $n$, the \emph{mobility exponent}, may be chosen as $n \in (0,3)$, as for $n \le 0$ the speed of propagation is infinite and non-negativity of $h$ is not ensured, whereas for $n \ge 3$ the boundary of the film cannot move (\emph{no-slip paradox})~\cite{dd.1974,hs.1971,m.1964}. Indeed, for $n \in \left(0,3\right)$, non-negativity of $h$ is preserved (cf.~\cite{bf.1990}) and fluid films move with finite and in general non-vanishing speed; see also Figure~\ref{fig01}.
\begin{figure}[htbp]
\psfrag{h}{$h(t,z)$}
\psfrag{x}{$z$}
\psfrag{t0}{$0<t\ll1$}
\psfrag{t1}{$t\gg1$}
	\centering
		\includegraphics[width=0.5\textwidth]{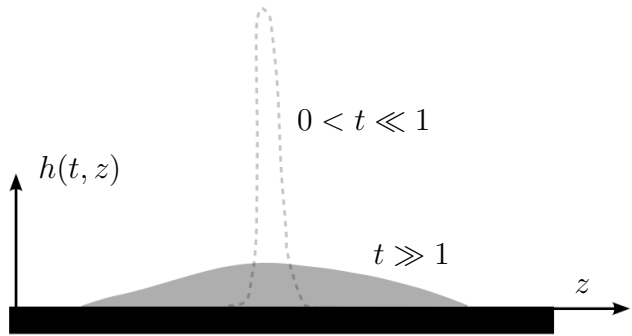}
		\caption{Sketch of the main geometry for a thin film (grey) on a solid (black). The source-type ($\delta$-distribution) initial condition is indicated after a short positive time evolution (dashed grey) while a droplet-type solution is expected to be reached in many cases after longer times (solid grey). We remark that the mass and center of mass are conserved.}
		\label{fig01}
\end{figure}  

Formally, one may view the thin-film equation \eqref{tfe} as an energy-driven flow
\begin{equation}\label{tfe_energy}
\partial_t h = \partial_z\left(M(h)\partial_z\left[\textnormal{D}_hE(h)\right]\right)
\end{equation}
where $E$ denotes the energy, $\textnormal{D}_h$ is the variational derivative, and $M(h)$ a general mobility. For viscous flow purely driven by \emph{surface tension} the basic energy functional is the arc length $\int \sqrt{1+(\partial_z h)^2}~\textnormal{d}z$, where integration is understood over the support of $h$. As the film is \emph{thin}, one uses an approximation of the energy functional and sets $E=\frac12\int (\partial_z h)^2~\textnormal{d}z$. Assuming $M(h)=h^n$ and using $E$ implies that~\eqref{tfe_energy} reduces to~\eqref{tfe}. More precisely, one may derive~\eqref{tfe} by means of a \emph{lubrication approximation}~\cite{go.2002,km.2013,km.2015} directly from the Navier-Stokes equations. Physical situations of interest, where thin-film flows occur are Darcy's flow in the Hele-Shaw cell with mobility exponent $n = 1$ or the Navier-Stokes equations with (non-)linear slip for general $n \in (0,3)$~(see {e.g.}~\cite{GoldsteinPesciShelley,AlmgrenBertozziBrenner}).

\medskip 

Problem~\eqref{source_problem} has been studied by Bernis, Peletier, and Williams in \cite{bpw.1992} under the assumption that solutions have the (mass-conserving) self-similar form
\begin{equation}\label{self-similar}
h(t,z) = t^{-\frac{1}{n+4}} H(Z) \quad \mbox{with } \, Z = t^{-\frac{1}{n+4}} z,
\end{equation}
such that $H \in C^1(\R) \cap C^3\left(\{H > 0\}\right) \cap L^1(\R)$ and $H^n \frac{\d^3 H}{\d Z^3} \in C^1\left(\{H > 0\}\right)$. Solutions to~\eqref{tfe} in general have finite speed of propagation~\cite{b.1996,b.1996.2} and this is in particular true for self-similar solutions of the form~\eqref{self-similar}. Then the solution has a free boundary, the contact line, where liquid, gas, and solid meet; see Figure~\ref{fig02}. The assumption $H \in C^1(\R)$ implies that the contact angle at this triple junction is zero, which is also known as \emph{complete wetting regime}, as the droplet generically wets the whole surface.
\begin{figure}[htbp]
\psfrag{air}{\scriptsize{gas}}
\psfrag{solid}{\scriptsize{solid}}
\psfrag{liquid}{\scriptsize{liquid}}
\psfrag{theta}{$\theta_0$}
	\centering
		\includegraphics[width=0.25\textwidth]{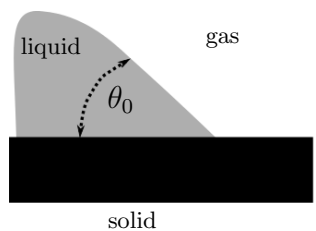}
		\caption{Sketch of the geometry near the contact line for a fixed time. The interface between solid (black), liquid (grey), and gas (white) is indicated. The contact angle $\theta_0$ is shown in this case for the situation with $\theta_0>0$. This figure can be viewed as a zoom near the right contact point for the droplet solution in Figure~\ref{fig01}.}
		\label{fig02}
\end{figure} 

We will assume a relaxed condition, i.e., 
\begin{equation}
H \in C^0(\R) \cap C^3\left(\{H > 0\}\right) \cap L^1(\R) \quad \text{and}\quad H^n \frac{\d^3 H}{\d Z^3} \in C^1\left(\{H > 0\}\right),
\end{equation}
that is, we also allow for nonzero contact angles. Note that this still corresponds to the complete wetting regime: The \emph{equilibrium microscopic contact angle} remains zero, but the \emph{dynamic macroscopic contact angle} is nonzero and relaxes in time; we refer to Hocking~\cite{h.1992} for a discussion of the difference between microscopic and macroscopic contact angles, and to more recent accounts in~\cite{re.2007,rhe.2010,re.2011,beimr.2009}. Nevertheless, the singularity of the solution at the boundary $\partial\{h > 0\}$ has very similar features as for general solutions with nonzero static microscopic contact angles (the \emph{partial wetting case}).
\medskip

The interest in the study of solutions of the form~\eqref{self-similar} arises from the insight that the long-time dynamics of compactly supported solutions to~\eqref{tfe} are commonly believed to be governed by the source-type self-similar profile. For $n = 1$, this was proven rigorously in~\cite{ct.2002} and refined results are contained for instance in~\cite{cu.2007,mms.2009,cu.2014,g.2015}. In particular in~\cite{g.2015}, one of the authors of this paper was able to prove that (in suitably chosen moving coordinates) arbitrarily many derivatives of $h$ converge to the source-type self-similar solution, which has the explicit form $H(Z) = C_1 \left(C_2 - Z^2\right)^2$ for $\verti{Z} \le C_2$ with $C_1, C_2 > 0$, first found by Smyth and Hill in~\cite{sh.1988}. For general mobility exponents, convergence was numerically observed by Bernoff and Witelski in~\cite{bw.2002}. A more recent numerical study was carried out by Peschka in~\cite{p.2015} with more precise control of the solution at the free boundary; in fact, there is a vast literature on numerical approaches to thin-film equations (see {e.g.}~\cite{BertozziMuenchFantonCazabat,DiezKondic,GruenRumpf,HerdeThieleHerminghausBrinkmann}).

\medskip

It is natural to consider the study of source-type self-similar solutions as a model problem for the treatment of general solutions to~\eqref{tfe}. In particular, the investigation of the aforementioned \emph{contact-line singularity} is considerably simpler than in the case of general solutions as techniques from ordinary differential equations (ODEs) are more directly applicable. The aim of this work is to demonstrate that we can treat this singularity by studying the invariant manifolds of suitable dynamical systems, as proposed and outlined in~\cite{ggo.2013,ggo.2015} by Giacomelli, Otto, and one of the authors of this paper, where the case $n \in \left(\frac 3 2, 3\right)$ for zero static contact angles was treated.

\medskip

The study of the regularity of solutions to~\eqref{tfe} at the free boundary has attracted increasing interest in the last years~\cite{ggko.2014,gk.2010,gko.2008,g.2015,j.2013,km.2015,k.2011,k.2015,km.2013}. Physically, the interest lies in a detailed understanding of the regularizing effect of various (nonlinear) slip conditions at the liquid-solid interface for the underlying fluid models. Mathematically, the boundary regularity is of interest as~\eqref{tfe} degenerates at $\partial\{h > 0\}$, i.e., uniform parabolicity is \emph{not} given. In fact, there are many related equations where uniform parabolicity is lost. This includes very recent topics, e.g., cross-diffusion models~\cite{JuengelKuehnTrussardi}, gradient flow approaches~\cite{DiFrancescoMatthes}, or approximations of the Boltzmann-Nordheim equation~\cite{JuengelWinkler}. The most classical case related to the thin-film equation is the \emph{second order porous-medium equation}, given by
\[
\partial_t h- \partial_z^2 (h^m) = 0 \quad \mbox{in } \, \{h > 0\},
\]
where $m > 1$ corresponds to $n-1$ in~\eqref{tfe}. However, the lack of a comparison principle for the \emph{fourth-order} thin-film equation makes the respective analysis more delicate. Compared to the porous-medium case, where solutions are smooth up to the boundary (cf.~Angenent~\cite{a.1988}), in the thin-film case -- unless $n = 1$ -- this is in general not true as proven for source-type solutions in~\cite{ggko.2014}, and for general solutions in a neighborhood of $n = 2$ in the complete wetting regime by Giacomelli, Kn\"upfer, Otto, and one of the authors of this paper in~\cite{ggko.2014,g.2015}. The partial wetting case was discussed by Kn\"upfer, and Kn\"upfer and Masmoudi, respectively, in~\cite{km.2015,k.2011,k.2015,km.2013,k.err.2016} covering the intervals $n \in \left(0,14/5\right) \setminus \{5/2, 8/3, 11/4\}$.

\medskip

The main results of this paper can be stated in a non-technical form as follows:
\begin{itemize}
 \item Consider the zero contact-angle case and mobility exponents $n\in(0,\frac32)$. Then the solution near the contact line can be expressed as a leading-order part with a remainder being real-analytic in suitable powers of the spatial variable (see Theorem~\ref{th:0n32_0}). The main insight of the proof is the existence of a degenerate real-analytic center manifold in suitable coordinates.
 \item For nonzero dynamic contact angles, we prove that for `most' mobility exponents in $n\in(0,3)$, the solution can again be expressed locally by a leading-order term with a real-analytic remainder in powers of the spatial variable (see Theorem~\ref{th:0n3_ne2}). 
 \item For certain mobility exponents ($n=3-1/m$ with $m\in\N$), we show that the real-analytic remainder contains also logarithmic terms (see Theorem~\ref{th:n2}). The reason is that resonances can occur for the flow on the center manifold.
 \item Furthermore, we prove existence results of certain global self-similar profiles using a shooting argument within a boundary manifold formulation and use a direct argument to establish uniqueness (cf.~\S\ref{sec:global} for details).
 \item We conclude the paper with a discussion of how our analysis transfers to radially-symmetric source-type self-similar solutions with zero contact angle in higher dimensions (cf.~App.~\ref{app:higher}).
\end{itemize}
In summary, we improve the understanding of contact-line singularities by identifying suitable degenerate and non-degenerate invariant manifolds and flows on these manifolds. One may conjecture that a similar program to study regularity near contact-type singularities may also work for other classes of non-uniformly parabolic equations, as well as other classes of free-boundary problems.

\medskip

The rest of the paper is structured as follows: In Section~\ref{sec:transformations} we state several useful preliminary transformations, which bring the thin-film equation into a more convenient form. Section~\ref{sec:results} contains the technical statements of our main results. The proofs for the local analysis near the contact line are given in Section~\ref{sec:local} while the global existence and uniqueness proofs are provided in Section~\ref{sec:global}.

\section{Transformations}\label{sec:transformations}
We use ansatz~\eqref{self-similar} in~\eqref{tfe}, which converts the partial differential equation (PDE), into a  fourth-order ODE
\[
\frac{\d}{\d Z} \left(H^n \frac{\d^3 H}{\d Z^3}\right) = \frac{1}{n+4} \frac{\d}{\d Z} \left(Z H\right) \quad \mbox{in } \, \{H > 0\}.
\]
This equation can be integrated once and we obtain
\begin{subequations}\label{problem_z}
\begin{equation}\label{ode_third}
H^{n-1} \frac{\d^3 H}{\d Z^3} = \frac{Z}{n+4} \quad \mbox{in } \, \{H > 0\}.
\end{equation}
The integration constant vanishes as we know that there exists a free boundary. This is proven in~\cite[\S 3]{bpw.1992} for zero contact angles and also transfers in an apparent way to the case of nonzero contact angles. In~\cite[Th.~1.2]{bpw.1992} it was shown that for $n \in (0,3)$ and zero contact angle, there exists an even non-negative solution to~\eqref{ode_third} with compact support and furthermore, being even is a necessary condition for $0 < n \le 2$. We will assume that $H = H(Z)$ is even from the outset. The free boundary is located at positions $\pm Z_0$ with $Z_0 > 0$. Then necessarily it follows that
\begin{equation}\label{bcz1}
H = 0 \quad \mbox{at } \, Z = \pm Z_0,
\end{equation}
which merely defines $Z_0$. The second condition determines the contact angle, i.e.,
\begin{equation}\label{bcz2}
\frac{\d H}{\d Z} = \mp \Theta \quad \mbox{at } \, Z = \pm Z_0,
\end{equation}
\end{subequations}
where $\Theta \ge 0$ is a constant. This fixes the time evolution of the slope of $h$ at the free boundary to be $\mp t^{-\frac{2}{n+4}} \Theta$.

\medskip

Problem~\eqref{problem_z} contains four boundary conditions~\eqref{bcz1} and~\eqref{bcz2} for the third-order ODE \eqref{ode_third}, thus seemingly being over-determined. Yet, by our symmetry assumption it suffices to impose $\frac{\d H}{\d Z} = 0$ at $Z = 0$ and by reflection to reduce problem~\eqref{problem_z} to solving
\begin{subequations}\label{problem_z_alt}
\begin{align}
H^{n-1} \frac{\d^3 H}{\d Z^3} &= \frac{Z}{n+4} \quad \mbox{for } \, Z \in \left(-Z_0,0\right),\label{ode_third_alt}\\
\left(H,\frac{\d H}{\d Z}\right) &= \left(0,\Theta\right) 
\quad \mbox{at } \, Z = - Z_0,\\
\frac{\d H}{\d Z} &= 0 \quad \mbox{at } \, Z = 0.
\end{align}
\end{subequations}
By rescaling $Z$ and $H$, we may assume that the pre-factor $(n+4)^{-1}$ in~\eqref{ode_third_alt} can be removed and the free boundary is located at $- Z_0 = -1$. Using the shifted variable $x := Z + 1$, we are thus lead to consider the problem
\begin{subequations}\label{problem_x}
\begin{align}
H^{n-1} \frac{\d^3 H}{\d x^3} &= -1+x \quad \mbox{for } \, x \in \left(0,1\right),\label{ode_third2}\\
\left(H,\frac{\d H}{\d x}\right) &= \left(0,\theta\right) \quad 
\mbox{at } \, x = 0,\label{bcx1}\\
\frac{\d H}{\d x} &= 0 \quad \mbox{at } \, x = 1,\label{bcx2}
\end{align}
\end{subequations}
with a constant $\theta \ge 0$. The rest of this paper will concentrate on the discussion of~\eqref{problem_x}.

\section{Main Results and Discussion}\label{sec:results}

\subsection{Notation}
For real-valued functions $f, g$ we write $f \lesssim g$ or $g \gtrsim f$, whenever a constant $C > 0$, only depending on the mobility exponent $n$, exists with $f \le C g$. A dependence of $C$ on other parameters $\omega$ is usually indicated by a subscript, i.e., we write $f \lesssim_\omega g$ or $g \gtrsim_\omega f$. We also say that a property $P$ is true for all $x \gg 1$ or $x \ll 1$, whenever a constant $C > 0$ exists such that $P$ is true for all $x \ge C$ or $x \le C$, respectively. A dependence on parameters is specified analogously. The symbol $\N = \{1,2,3,\ldots\}$ denotes the natural numbers and we write $\N_0 := \N \cup \{0\}$.

\subsection{Main Results}
Existence and uniqueness of solutions to \eqref{problem_x} for $\theta = 0$ in the class $H \in C^1([0,1]) \cap C^3((0,1))$ were proven in~\cite{bpw.1992} and the boundary regularity for $n \in (3/2,3)$ was addressed and discussed exhaustively in \cite{ggo.2013}. Indeed, there it was shown that
\[
H = \mu x^{\frac 3 n} \left(1 + v\left(x,x^\beta\right)\right) \quad \mbox{for } \, 0 \le x \ll 1,
\]
where $\mu > 0$ is an $n$-dependent constant, $\beta \in (0,1)$ is an in general irrational $n$-dependent number, and $v = v(x_1,x_2)$ is an analytic function in a neighborhood of $(x_1,x_2) = (0,0)$. The leading-order asymptotic $H = \mu x^{\frac 3 n} (1 + o(1))$ as $x \searrow 0$ was already proven in~\cite{bpw.1992}. For $n \in (0,3/2)$ this asymptotic is given by
\begin{equation}\label{bpw_0n32}
H = \kappa x^2 (1 + o(1)) \quad \mbox{as } \, x \searrow 0,
\end{equation}
with another $n$-dependent constant $\kappa > 0$. Here, we are able to show:
\begin{theorem}\label{th:0n32_0}
Suppose $n \in (0,3/2)$, $\theta = 0$, and $H>0$ for $x\in(0,1]$. Then the unique solution $H$ of problem~\eqref{problem_x} fulfills the asymptotic
\begin{equation}\label{power_0n32}
H = \kappa x^2 \left(1 + v\left(\kappa^{-n} x^{\beta+1},\kappa^{-n}x^\beta\right)\right) \quad \mbox{for } \, 0 \le x \ll 1,
\end{equation}
where $\kappa > 0$ is an $n$-dependent constant, $\beta = 3 - 2 n$, and $v = v(x_1,x_2)$ is an analytic and $\kappa$-independent function in a neighborhood of $(x_1,x_2) = (0,0)$, $v(0,0)=0$ but $v\not\equiv 0$ with an in general non-trivial dependence on both $x_1$ and $x_2$.
\end{theorem}
For $n \in \left(1,3/2\right)$ this result was suggested by Giacomelli, Otto, and the second of the authors of this paper in~\cite{ggo.2013} and proven by the second of the authors of this paper in an appendix of his PhD thesis. The methods used here simplify this reasoning through the study of invariant manifolds of a suitable dynamical system as detailed below. We further mention that for $n = 1$ the ODE \eqref{ode_third2} is the linear one $\frac{\d^3 H}{\d x^3} = -1 + x$ for $x \in (0,1)$ and the solution of problem~\eqref{problem_x} can be explicitly computed to be $ H = \frac{1}{24} x^2 (2-x)^2$ for $x \in [0,1]$, the \emph{Smyth-Hill profile}~\cite{sh.1988}. For $n \in \left(0,3/2\right) \setminus \{1\}$, we may insert the expansion \eqref{power_0n32} into \eqref{ode_third2} and obtain
\begin{align*}
\partial_{x_1} v(0,0) &= - \left((5-2n) (4-2n) (3-2n)\right)^{-1}, \\
\partial_{x_2} v(0,0) &= \left((6-2n) (5-2n) (4-2n)\right)^{-1},
\end{align*}
proving the non-trivial dependence on both $x_1$ and $x_2$.

\medskip

For $\theta > 0$, necessarily $H = \theta x (1+o(1))$ as $x \searrow 0$ due to the boundary condition \eqref{bcx1}. For the correction we are able to prove:
\begin{theorem}\label{th:0n3_ne2}
For given $n \in \left(0,3\right) \setminus \{3-\frac1m:m\in\N\}$ and $\theta > 0$ there exists a unique solution $H \in C^1([0,1]) \cap C^3((0,1))$ of problem~\eqref{problem_x} with the asymptotic
\begin{equation}\label{power_0n3}
H = \theta x \left(1 + v\left(x,x^\beta\right)\right) \quad \mbox{for } \, 0 \le x \ll 1,
\end{equation}
where $\beta = 3 - n$ and $v = v(x_1,x_2)$ is an analytic function in a neighborhood of $(x_1,x_2) = (0,0)$, $v(0,0)=0$ but $v\not\equiv 0$  with an in general non-trivial dependence on both $x_1$ and $x_2$.
\end{theorem}
Again, for $n = 1$ we can calculate the explicit solution
\begin{equation}\label{n1_explicit}
H = \frac{1}{24} x(x-2)(x(x-2)-12 \theta ) \quad \mbox{for} \quad x \in [0,1].
\end{equation}
For general  $n \in \left(0,3\right) \setminus \{3-\frac1m:m\in\N\}$, we may use the expansion \eqref{power_0n3} in \eqref{ode_third2} and obtain explicitly
\begin{align*}
\partial_{x_1} v(0,0) &= b,\\
\partial_{x_2} v(0,0) &= \left((4-n) (3-n) (2-n)\right)^{-1} \theta^{-n},\\
\partial_{x_1} \partial_{x_2} v(0,0) &= \left(1 + (1-n) b\right) \left((5-n) (4-n) (3-n)\right)^{-1} \theta^{-n}
\end{align*}
with $b \in \R$, demonstrating that the expansion in both $x_1$ and $x_2$ is non-trivial.

\medskip

The case $n = 3 - \frac 1m$, where $m \in \N$, turns out to be particular as here we can show:
\begin{theorem}\label{th:n2}
For $n = 3 - \frac 1m$, where $m \in \N$, and $\theta > 0$ there exists a unique solution $H \in C^1([0,1]) \cap C^3((0,1))$ of problem~\eqref{problem_x} with the asymptotic
\begin{equation}\label{power_n2}
H = \theta x \left(1 + v\left(x\log x,x^\beta\right)\right) 
\quad \mbox{for } \, 0 \le x \ll 1,
\end{equation}
where $\beta = 3 - n = \frac 1 m$ and $v = v(x_1,x_2)$ is an analytic function in a neighborhood of $(x_1,x_2) = (0,0)$, $v(0,0)=0$ but $v\not\equiv 0$ with an in general non-trivial dependence on both $x_1$ and $x_2$.
\end{theorem}
Inserting the expansion \eqref{power_n2} into \eqref{ode_third2} for $n = 5/2$, we obtain the explicit expansion
\[
H = \theta x \left(1 + \frac{8}{3} \theta^{- \frac 5 2} x^{\frac 1 2} - 2 \theta^{-5} x \log x + \cO(x)\right) \quad \mbox{as} \quad x \searrow 0,
\]
that is, logarithmic terms do appear in general.

\subsection{Discussion}
The rest of the paper will focus on the proofs of Theorems~\ref{th:0n32_0}, \ref{th:0n3_ne2}, and \ref{th:n2}. The key ingredient is the reduction of the dynamics close to the contact line (that is, $x = 0$) to the study of invariant manifolds of suitable dynamical systems and the investigation of the eigenvalues of a linearization. This determines the exponents in the power series~\eqref{power_0n32} and~\eqref{power_0n3}. For example, for $n = 2$ and $\theta > 0$ two eigenvalues of the linearization coalesce leading to a nontrivial Jordan block. This effect leads to the logarithm in~\eqref{power_n2} at the first resonance. The existence and uniqueness part in Theorems~\ref{th:0n3_ne2} and~\ref{th:n2} is then obtained by means of a shooting argument, matching the symmetry condition \eqref{bcx2} at $x = 1$.

\medskip

We note that the expansion~\eqref{power_0n3} has been found by Kn\"upfer in~\cite{k.2015,k.err.2016} for $n \in (0,14/5) \setminus \left\{2,\frac 5 2, \frac 8 3, \frac{11}{4}\right\}$ and the expansion~\eqref{power_n2} for $n = 2$ by the same author in~\cite{k.2011} for general solutions close to a stationary profile with nonzero static contact angle, relying on a considerably greater technical effort than in this note. However, Kn\"upfer's result \cite{k.2015,k.err.2016} does not capture possible resonances at $n = 5/2$, $n = 8/3$, and $n = 11/4$. It should also be pointed out that the resonances at certain values of $n$ appear as points of special technical difficulty in~\cite{k.2015}, namely as a validity boundary ($n=14/5$ cf.~\cite[Thm.~4.1]{k.2015}) and in the construction of certain weighted Sobolev spaces ($n=5/2$ cf.~\cite[eq.~(45)]{k.2015} and $n=11/4$ cf.~\cite[eq.~(47)]{k.2015}), as well as other constraints stated in \cite{k.err.2016}.

\medskip

Our motivation for partially re-visiting these results lies primarily in the simplicity of our arguments and in a thorough understanding of the sequence of resonances on approaching the no-slip case $n = 3$, stated in Theorem~\ref{th:n2}. It appears that local analysis via dynamical systems methods shows more clearly where certain expansion powers and coefficients arise, why logarithmic terms must appear, and why global existence and uniqueness follow from certain nonlinear sign/monotonicity properties of the equation.

\section{Local analysis at the contact line}\label{sec:local}
\subsection{Zero contact angles}\label{ssec:cw}

Here we consider the zero-contact angle case with $n \in (0,3/2)$. The goal is to prove the local statements near the contact line stated in Theorem~\ref{th:0n32_0}. Due to the result of~\cite{bpw.1992} ({cf.}~\eqref{bpw_0n32}), we may set
\begin{equation}\label{trafo_0n32}
H =: x^2 F \quad \mbox{where } \, 0 \le x \ll 1.
\end{equation}
Using~\eqref{trafo_0n32}, equation~\eqref{ode_third} transforms into
\begin{equation}\label{eq_0n32}
F^{n-1} x^{2 n - 3} q\left(x \frac{\d}{\d x}\right) F = -1 + x 
\quad \mbox{for } \, 0 < x \ll 1,
\end{equation}
where
\begin{equation}\label{def_q}
q\left(x \frac{\d}{\d x}\right) = x \frac{\d^3}{\d x^3} x^2 = \left(x \frac{\d}{\d x}\right)^3 + 3 \left(x \frac{\d}{\d x}\right)^2 + 2 x \frac{\d}{\d x} = x \frac{\d}{\d x} \left(x \frac{\d}{\d x}+1\right) \left(x \frac{\d}{\d x}+2\right)
\end{equation}
is a polynomial operator of order 3 in $x \frac{\d}{\d x}$. It is also convenient to introduce $x_1 := x^{3 - 2 n}$, $x_2 := x^{4-2n}$, and to pass to (natural) $\log$-coordinates
\begin{equation}\label{log_coord}
s := \log x,
\end{equation}
as this results in $\frac{\d}{\d s} = x \frac{\d}{\d x}$, whence~\eqref{eq_0n32} turns into
\begin{equation}\label{eq_0n32_s}
q\left(\frac{\d}{\d s}\right) F = \frac{- \txte^{(3-2n)s} + \txte^{(4-2n)s}}{F^{n-1}} \quad \mbox{for } \, - \infty < s \ll -1.
\end{equation}
Using the notation
\begin{equation}\label{not_0n32}
x_1 := \txte^{(4-2n) s}, \quad x_2 := \txte^{(3-2n) s}, \quad F^\prime := \frac{\d F}{\d s}, \quad \mbox{and } \, F^{\prime\prime} := \frac{\d^2 F}{\d s^2},
\end{equation}
we can turn \eqref{eq_0n32_s} into an autonomous five-dimensional dynamical system
\begin{equation}\label{dyn_0n32}
\frac{\d}{\d s} \begin{pmatrix} x_1 \\ x_2 \\ F \\ F^\prime \\ F^{\prime\prime} \end{pmatrix} = \mathcal F\left(x_1,x_2,F,F^\prime,F^{\prime\prime}\right) := \begin{pmatrix} (4-2n) x_1 \\ (3-2n) x_2 \\ F^\prime \\ F^{\prime\prime} \\ \frac{x_1 - x_2}{F^{n-1}} - 3 F^{\prime\prime} - 2 F^\prime \end{pmatrix} \quad \mbox{for } \, - \infty < s \ll 1.
\end{equation}
The next step is to describe the leading asymptotics as $s \to - \infty$ as this corresponds to understanding the boundary value problem near the contact line for $0<x\ll1$.
\begin{lemma}\label{lem:conv_0n32}
Consider the ODE~\eqref{dyn_0n32}. Then trajectories, which satisfy the thin-film equation~\eqref{problem_x} after transformation, must satisfy
\begin{equation}\label{as_f_summary}
\left(x_1, x_2, F, F^\prime, F^{\prime\prime}\right) \to p_\kappa := 
\left(0,0,\kappa,0,0\right) \quad \mbox{as } \, s \to - \infty
\end{equation}
for some constant $\kappa>0$ given in~\eqref{bpw_0n32}.
\end{lemma}
\begin{proof}
As $s \to - \infty$ necessarily $F \to \kappa$ with $\kappa > 0$ by using~\eqref{bpw_0n32} and the definition~\eqref{trafo_0n32}. By~\eqref{not_0n32} it follows that $x_1 \to 0$ as well as $x_2 \to 0$. We claim that $F^\prime \to 0$ and $F^{\prime\prime} \to 0$ as well. Therefore, we repeat the arguments in~\cite[\S 5]{bpw.1992} and consider $\psi := \left(\frac{\d H}{\d x}\right)^2$ as a function of $H$ (invertibility is ensured for $0 \le x \ll 1$, respectively $0 \le H \ll 1$, by~\eqref{bpw_0n32}). Thus~\eqref{ode_third} transforms into
\begin{equation}\label{eq_psi}
\frac{\d^2 \psi}{\d H^2} = 2 \psi^{- \frac 1 2} H^{1-n} \left(-1+x\right) 
\quad \mbox{as } \, H,x \searrow 0.
\end{equation}
Since $x = x(H) \to 0$ as $H \searrow 0$, Taliaferro's result~\cite{t.1979} implies $\psi = 4 \kappa H (1 + o(1))$ as $H \searrow 0$, which, by using equation~\eqref{eq_psi} and integrating once, upgrades to
\begin{equation}\label{as_psi}
\psi = 4 \kappa H (1 + o(1)), \quad \frac{\d \psi}{\d H} = 4 \kappa (1 + o(1)), 
\quad \mbox{and } \, \frac{\d^2 \psi}{\d H^2} = 
\kappa^{-\frac 1 2} H^{\frac 1 2 - n} (1 + o(1)) \quad \mbox{as } \, H \searrow 0.
\end{equation}
The asymptotics~\eqref{as_psi} can be transformed into asymptotics for $H$, which read
\begin{equation}\label{as_H}
H = \kappa x^2 (1 + o(1)), \quad \frac{\d H}{\d x} = 2 \kappa x (1 + o(1)), 
\quad \mbox{and } \, \frac{\d^2 H}{\d x^2} = 2 \kappa(1 + o(1)) \quad 
\mbox{as } \, x \searrow 0.
\end{equation}
Using~\eqref{trafo_0n32} and~\eqref{not_0n32}, it is straight-forward to see that this implies $F^\prime \to 0$ and $F^{\prime\prime} \to 0$ as $s \to - \infty$.
\end{proof}
The point $p_\kappa$ is a steady state of the dynamical system \eqref{dyn_0n32} as clearly $\mathcal F (p_\kappa) = 0$. Since we are interested in the asymptotic behavior and regularity of trajectories of~\eqref{dyn_0n32}, which converge to $p_\kappa$ as $s\ra -\I$, we must check the structure and regularity of invariant unstable and/or center-unstable manifolds; for background on stable and unstable invariant manifold theory for dynamical systems we refer to~\cite{GH,Kelley1} and for center manifolds to~\cite{Carr,Kelley1,Sijbrand}. Furthermore, we remark that we are only going to be interested in certain manifolds locally near a steady state, so formally we should write, e.g., a stable manifold as $W^\txts_{\textnormal{loc}}(\cdot)$. However, we shall omit the subscript in this case for notational simplicity.
\begin{proposition}\label{prop:steady_0n32}
For each fixed $\kappa>0$, the steady state $p_\kappa$ is non-hyperbolic with two-dimensional stable, one-dimensional center, and two-dimensional unstable manifolds. Any trajectory converging to $p_\kappa$ as $s\ra -\I$ lies on the unstable manifold and satisfies
\begin{equation}\label{eq:asymp1}
F = \kappa \left(1+v\left(\kappa^{-n} x_1, \kappa^{-n} x_2\right)\right)
\end{equation}
for some real-analytic and $\kappa$-independent function $v$ and $-\I<s\ll_{\kappa} -1$.
\end{proposition}
\begin{proof}
The linearization of $\mathcal F$ at the steady state is given by
\[
\txtD \mathcal F(p_\kappa) = \begin{pmatrix} 4-2n & 0 & 0 & 0 & 0 \\ 0 & 3-2n & 0 & 0 & 0 \\ 0 & 0 & 0 & 1 & 0 \\ 0 & 0 & 0 & 0 & 1 \\ \kappa^{1-n} & - \kappa^{1-n} & 0 & -2 & -3 \end{pmatrix}
\]
and the characteristic polynomial can be easily computed to be
\begin{equation}\label{char_0n32}
\zeta\mapsto \left(\zeta - (4-2n)\right) \left(\zeta - (3-2n)\right) q(\zeta),
\end{equation}
where $q(\zeta) = \zeta (\zeta + 1) (\zeta + 2)$ is the same as in~\eqref{def_q}. Since we are dealing with the case $n\in(0,3/2)$, there are two positive eigenvalues, $4-2n$ and $3-2n$, and two negative ones, namely $-1$ and $-2$. However, as the single eigenvalue $0$ appears, we are \emph{not} dealing with a hyperbolic steady state. Nonetheless, note that the eigenvector associated to $0$ is given by $(0,0,1,0,0)^\top$ so that the linear center eigenspace is
\[
E^\txtc(p_\kappa)=\{x_1=0,x_2=0,F^\prime=0,F^{\prime\prime}=0\}.
\]
Furthermore, $E^\txtc(p_\kappa)$ is an invariant subspace for the full nonlinear ODE~\eqref{dyn_0n32} as it consists entirely of equilibrium points. Therefore, for fixed $\kappa>0$, the center manifold $W^\txtc(p_\kappa)$ coincides with $E^\txtc(p_\kappa)$. In fact, $W^\txtc(p_\kappa)$ is obviously analytic as it is just a coordinate axis. We remark that this is a special case as \emph{general} center manifolds are not always analytic~\cite{Sijbrand} but the analyticity is still covered by general theory on certain analytic classes of center manifolds as discussed in~\cite{Aulbach}. Since $W^\txtc(p_\kappa)$ consists entirely of equilibria, any sufficiently small tubular neighborhood $\cT$ of it is foliated by the two-dimensional stable and unstable manifolds $W^\txts(p_\kappa)$ and $W^\txtu(p_\kappa)$. Since $0\leq x\ll1$, we only have to consider a solution $\gamma(s)$ as $s\ra -\infty$ which has to converge to a single equilibrium $p_\kappa$. Due to the local foliation around $W^\txtc(p_\kappa)$, this implies that $\gamma(s) \in W^\txtu(p_\kappa)$ for all $s\in(-\I,-s_0]$ for some sufficiently large $s_0$. Therefore, the regularity and asymptotic expansion of $\gamma$ can be studied considering the regularity and dynamics on $W^\txtu(p_\kappa)$. Since $\kappa>0$, we know that $\cF$ is analytic near $p_\kappa$ and this implies analyticity of $W^\txtu(p_\kappa)$~\cite[p.~330]{CoddingtonLevinson}. Analyticity implies
\[
W^\txtu(p_\kappa)=\{(a,b)\in\R^2\times \R^3:b=g(a)\}\cap \cT
\]
where $a,b$ are suitable coordinates and $g:\R^2\ra \R^3$ is an analytic mapping locally inside $\cT$. Denote the unstable eigenvalues by $\lambda_1=4-2n$ and $\lambda_2=3-2n$. The associated eigenvectors are easily computed as 
\begin{align}
V_1 &= \left(- (n-2) (2 n - 5) (n-3) \kappa ^{n-1},0,\frac{1}{4},\frac{2-n}{2},(n-2)^2\right)^\top\label{eq:evec2}, \\
V_2 &= \left(0,2 (2 n-3) (n-2) (2n-5) \kappa^{n-1},1,3-2n,(2n-3)^2\right)^\top.
\label{eq:evec1}
\end{align}
Since $\left(V_1, V_2\right)$ span a plane non-orthogonal to the $(x_1,x_2)$-plane, the tangent space of $\cF$ in $p_\kappa$ can be parametrized by the coordinates $a = (x_1,x_2)$ with the remaining coordinates $b=(F,F',F'')$. Thus we can write the unstable manifold $W^\txtu(p_\kappa)$ as an analytic graph of $(x_1,x_2)$ in a neighborhood of $p_\kappa$. In particular, the first component of this graph is an analytic function $g_1 = g_1(x_1,x_2) = \kappa\left(1 + v_\kappa(x_1,x_2)\right)$. In order to eliminate the $\kappa$-dependence of $v_\kappa$, we notice that \eqref{dyn_0n32} is invariant with respect to the scaling
\[
\left(x_1, x_2, F, F^\prime, F^{\prime\prime}\right) \mapsto \left(\kappa^n x_1, \kappa^n x_2, \kappa F, \kappa F^\prime, \kappa F^{\prime\prime}\right) \quad \mbox{for any} \quad \kappa > 0,
\]
so~\eqref{eq:asymp1} follows.
\end{proof}
Just transforming back to the coordinate $x$ via~\eqref{not_0n32} yields
\[
F = \kappa(1+v(\kappa^{-n} x_1, \kappa^{-n} x_2))=\kappa(1+v(\kappa^{-n} x^{\beta+1},\kappa^{-n} x^\beta)), 
\]
where $v$ is an analytic function and the local result~\eqref{power_0n32} in Theorem~\ref{th:0n32_0} follows. Furthermore, one may be interested in whether the mapping $v$ is non-trivial. Obviously, $v(0,0)=0$ since the third coordinate of the steady state is $\kappa$. 
\begin{lemma}
\label{lem:ntcw}
The function $v$ is non-trivial. In particular, we have
\begin{equation}
\frac{\partial v}{\partial x_1}(0,0)=\frac{(V_1)_3}{(V_1)_2} \ne 0\quad \text{and} \quad \frac{\partial v}{\partial x_2}(0,0)=\frac{(V_2)_3}{(V_2)_1} \ne 0.
\end{equation}
where $(V_i)_j$ is the $j$-th component of the $i$-th eigenvector from~\eqref{eq:evec1}.
\end{lemma}
\begin{proof}
The unstable eigenspace $E^\txtu(p_{1})$ (we may focus on the case $\kappa = 1$) is spanned by $V_1$, $V_2$ so
\begin{equation}
\label{eq:Euup}
E^\txtu(p_{1})=\{(x_1,x_2,F,F',F'')^\top=c_1 V_1 + c_2 V_2 + (0,0,1,0,0)^\top:c_1,c_2\in\R\}. 
\end{equation}
Using the first two equations giving $E^\txtu(p_{1})$ in~\eqref{eq:Euup}, we may express $c_1,c_2$ in terms of $x_1,x_2$. $E^\txtu(p_{1})$ is the tangent space to $W^\txtu(p_{1})$ so differentiating $F$ with respect to $x_1$ and $x_2$ and evaluating each time at $(x_1,x_2)=(0,0)$ yields the result.
\end{proof}
The last result also shows the difficulty to extend the same method beyond the interval $n\in(0,\frac32)$ as $(V_1)_1$ vanishes for $n=\frac32$. In fact, a dynamical systems approach is in general a very helpful strategy to identify special parameter values. Here these values are those of the exponent $n$.

\medskip

The following corollary turns out to be convenient for the proof of existence in Section~\ref{sec:globalexist}.
\begin{corollary}\label{cor:exp_0n32}
Let $H$ denote the solution to \eqref{ode_third2}\&\eqref{bcx1} with $\theta = 0$. Then
\begin{subequations}\label{hdhd2h_0n32}
\begin{align}
H &= \kappa x^2 \left(1 + \cO\left(\kappa^{-n} x^{3-2n} + \kappa^{-n} x^{4-2n}\right)\right) \quad \mbox{as} \quad x \searrow 0,\label{hdhd2h_0n32_0}\\
\frac{\d H}{\d x} &= 2 \kappa x \left(1 + \cO\left(\kappa^{-n} x^{3-2n} + \kappa^{-n} x^{4-2n}\right)\right) \quad \mbox{as} \quad x \searrow 0,\label{hdhd2h_0n32_1}\\
\frac{\d^2 H}{\d x^2} &= 2 \kappa \left(1 + \cO\left(\kappa^{-n} x^{3-2n} + \kappa^{-n} x^{4-2n}\right)\right) \quad \mbox{as} \quad x \searrow 0, \label{hdhd2h_0n32_2}
\end{align}
\end{subequations}
where $\kappa > 0$.
\end{corollary}
\begin{proof}
We may use expansion \eqref{eq:asymp1} of Proposition~\ref{prop:steady_0n32} (cf.~\eqref{trafo_0n32}):
\begin{subequations}\label{exp_hdhd2h_kappa}
\begin{equation}\label{exp_h_kappa}
H = \kappa x^2 \left(1 + v\left(\kappa^{-n} x_1, \kappa^{-n} x_2\right)\right) \quad \mbox{with} \quad (x_1,x_2) = \left(x^{3-2n}, x^{4-2n}\right),
\end{equation}
where $\kappa^{-n} (\verti{x_1} + \verti{x_2}) \ll 1$ and $v$ is $\kappa$-independent and analytic in a neighborhood of $(x_1,x_2) = (0,0)$ with $v(0,0) = 0$. Differentiating \eqref{exp_h_kappa} with respect to $x$, we obtain
\begin{align}
\begin{split}
\frac{\d H}{\d x} &= \kappa x \left(2 + \left(x \frac{\d}{\d x} + 2\right) v\left(\kappa^{-n} x_1, \kappa^{-n} x_2\right)\right) \\
&= \kappa x \left(2 + \left(x_1 \partial_{x_1} v + x_2 \partial_{x_2} v + 2 v\right)\left(\kappa^{-n} x_1, \kappa^{-n} x_2\right)\right),
\end{split}\\
\begin{split}
\frac{\d^2 H}{\d x^2} &= \kappa \left(2 + q \left(x \frac{\d}{\d x}\right) v\left(\kappa^{-n} x_1, \kappa^{-n} x_2\right)\right) \\
&= \kappa \left(2 + \left(q\left(x_1 \partial_{x_1} + x_2 \partial_{x_2}\right)v\right)\left(\kappa^{-n} x_1, \kappa^{-n} x_2\right)\right)
\end{split}
\end{align}
\end{subequations}
with $q(\zeta) = (\zeta+1)(\zeta+2)$, $(x_1,x_2) = \left(x^{3-2n}, x^{4-2n}\right)$, and $\kappa^{-n} (\verti{x_1} + \verti{x_2}) \ll 1$. Equations~\eqref{exp_hdhd2h_kappa} and $v(0,0) = 0$ immediately yield \eqref{hdhd2h_0n32}.
\end{proof}

\subsection{Nonzero dynamic contact angles: setup}\label{ssec:pw}
Now we proceed to the case of nonzero contact angles. The first step is again to reformulate the problem as a suitable autonomous dynamical system. The main strategy will be as in Section~\ref{ssec:cw} to try to identify invariant manifolds. However, a changed dynamical system has to be analyzed which leads to considerable differences, so we present all calculations.

\medskip

We start with the ansatz
\begin{equation}\label{trafo_pw_0n3}
H =: x F \quad \mbox{where } \, 0 \le x \ll 1.
\end{equation}
to factorize out the leading-order. Using~\eqref{trafo_pw_0n3}, equation~\eqref{ode_third2} transforms into
\begin{equation}\label{eq_pw_0n3}
x^{n-3} F^{n-1} \tilde{q}\left(x \frac{\d}{\d x}\right) F 
= -1 + x \quad \mbox{for } \, 0 < x \ll 1,
\end{equation}
where
\begin{equation}\label{def_qt}
\tilde{q}\left(x \frac{\d}{\d x}\right) = x^2 \frac{\d^3}{\d x^3} x 
= x \frac{\d}{\d x}\left(x \frac{\d}{\d x}-1\right)\left(x \frac{\d}{\d x}+1\right)
\end{equation}
is again a polynomial operator of order 3 in $x \frac{\d}{\d x}$. Now one just introduces slightly different coordinates $x_1 := x$, $x_2 := x^{3-n}$ and also passes to $\log$-coordinates $s := \log x$ (implying $x_1 = \txte^{s}$, $x_2 = \txte^{(3-n) s}$), so that
\begin{equation}\label{eq_0n3_s}
\tilde{q}\left(\frac{\d}{\d s}\right) F = \frac{x_2 (x_1-1)}{F^{n-1}} \quad 
\mbox{for } \, - \infty < s \ll -1.
\end{equation}
Using $F^\prime := \frac{\d F}{\d s}$ and $F^{\prime\prime} := \frac{\d^2 F}{\d s^2}$ as before, we re-write \eqref{eq_0n3_s} as an autonomous five-dimensional ODE
\begin{equation}\label{dyn_0n3}
\frac{\d}{\d s} \begin{pmatrix} x_1 \\ x_2 \\ F \\ F^\prime \\ F^{\prime\prime} 
\end{pmatrix} = \tilde{\mathcal F}\left(x_1,x_2,F,F^\prime,F^{\prime\prime}\right) 
:= \begin{pmatrix} x_1 \\  (3-n)x_2 \\ F^\prime \\ F^{\prime\prime} \\ 
\frac{x_2 (x_1-1)}{F^{n-1}} +F^\prime \end{pmatrix} \quad \mbox{for } \, 
- \infty < s \ll 1.
\end{equation}
Similar to the zero-contact angle case, the point $p_{\theta}:=(0,0,{\theta},0,0)$ is now a steady state of \eqref{dyn_0n3}. For the relevant regularity and asymptotics we must again study
\[
\left(x_1, x_2, F, F^\prime, 
F^{\prime\prime}\right) \to \left(0,0,{\theta},0,0\right) \qquad 
\text{as $s \to - \infty$.}
\]
The next result shows already some of the key differences between zero- and nonzero contact angle cases.
\begin{lemma}
Fix ${\theta}>0$ and consider~\eqref{dyn_0n3}. Then
\begin{equation}
\dim W^\txtu(p_{\theta}) = 3,\qquad \dim W^\txts(p_{\theta}) = 1, 
\qquad \dim W^\txtc(p_{\theta}) = 1. 
\end{equation}
Furthermore, for $n=2$ the matrix $\txtD \tilde{\mathcal F}(p_{\theta})$ is not diagonalizable, while for $n\neq 2$ it is diagonalizable.
\end{lemma}
\begin{proof}
The linearization of $\tilde{\mathcal F}$ at $p_{\theta}$ is
\begin{equation}\label{def_ta}
\tilde{A}:=\txtD \tilde{\mathcal F}(p_{\theta}) = \begin{pmatrix} 1 & 0 & 0 & 0 & 0 \\ 0 & 3-n & 0 & 0 & 0 \\ 0 & 0 & 0 & 1 & 0 \\ 0 & 0 & 0 & 0 & 1 \\ 0 & -{\theta}^{1-n} & 0 & 1 & 0 \end{pmatrix}
\end{equation}
with eigenvalues
\[
\tilde{\lambda}_1=1,\quad \tilde{\lambda}_2=3-n, \quad \tilde{\lambda}_3=1, \quad \tilde{\lambda}_4=0, \quad\tilde{\lambda}_5=-1.
\]
It is easy to check that $\tilde{A}$ is diagonalizable if $n\neq 2$. However, for $n=2$ there exists a triple eigenvalue $\tilde\lambda_1 = \tilde\lambda_2 = \tilde\lambda_3 = 1$. In this case, the Jordan canonical form for $\tilde{A}$ has a block of the form 
\begin{equation}\label{Jblock}
B:=\begin{pmatrix} 1 & 1 \\ 0 & 1 \end{pmatrix}.
\end{equation}
In both cases, the dimensions of the stable-, unstable-, and center manifolds follow from the eigenvalue configuration.
\end{proof}
The first main difference is that now the unstable manifold is three-dimensional. The second issue is illustrated by the Jordan block~\eqref{Jblock}. It is an indication that one must be aware of potential logarithmic terms. Indeed, if one solves a system of the form $\frac{\txtd \zeta}{\txtd s} = B \zeta$ for $\zeta=\zeta(s)\in\R^2$, then this yields
\begin{equation}
\label{eq:linsolven2}
\zeta_1(s)=c_1 \txte^s+c_2 s \txte^s,\qquad \text{for constants $c_1,c_2\in\R$.} 
\end{equation}
However, since $s = \log x$ by construction, we get $\zeta_1=c_1 x+c_2 x \log x$. Therefore, already on the linear level we see for $n=2$ a typical resonance effect. Presumably our approach is one of the easiest ways to see dynamically, why certain logarithmic expansion terms may be relevant for the thin-film equation near the contact line.
\begin{lemma}
\label{lem:umfldequat}
For $n\in(0,3)$ the unstable manifold $W^\txtu(p_{\theta})$ can locally be written as the graph of a mapping $\tilde{g}:\R^3\ra \R^2$, which is locally analytic in a tubular neighborhood of the center manifold. In particular, we have
\begin{equation}\label{eq_f_n_ne_2}
F'=\tilde{g}_1(x_1,x_2,F)=\tilde{g}_1(x,x^{3-n},F),
\end{equation}
where $\tilde{g}_1$ is locally analytic and $\tilde{g}_1(0,0,\theta)=0$. 
\end{lemma}
\begin{proof}
First, note that the center manifold is again a one-dimensional line of steady states, so we proceed in a similar manner as in Section~\ref{ssec:cw}. However, we use a case distinction for the mobility exponents and start with the case $n\neq 2$. In this case, the same arguments as in Section~\ref{ssec:cw} still apply with suitable modifications. We have to consider the, now three-dimensional, unstable manifold $W^\txtu(p_{\theta})$ and the associated tangent space $E^\txtu(p_{\theta})$ spanned by
\begin{align*}
\tilde{V}_1&= \left(1,0,0,0,0\right)^\top, \\
\tilde{V}_2&= \left(0, (n-2) (n-3) (n-4) {\theta} ^{n-1},1,3-n,(3-n)^2\right)^\top, \\
\tilde{V}_3&= (0,0,1,1,1)^\top.
\end{align*}
Therefore, the tangent space is determined by the set of equations
\begin{subequations}\label{es_n_ne_2}
\begin{align}
F^\prime &= - \frac{{\theta}^{1-n}}{(n-3)(n-4)} x_2 + F - \theta, \\
F^{\prime\prime} &= \frac{{\theta}^{1-n}}{n-3} x_2 + F - {\theta}.
\end{align}
\end{subequations}
It is now immediate from~\eqref{es_n_ne_2} that we can parametrize $W^\txtu(p_{\theta})$ by $(x_1,x_2,F)$ if $n\neq 2$. From~\eqref{es_n_ne_2} we can derive
\begin{equation}\label{eq_f_n_ne_2a}
x\frac{\txtd F}{\txtd x}=\frac{\txtd F}{\txtd s}=F'=\tilde{g}_1(x_1,x_2,F)
=\tilde{g}_1(x,x^{3-n},F).
\end{equation}
However, note that for $n=2$, the linearization is given by~\eqref{def_ta} and the eigenvalues are $\tilde{\lambda}_{1,2,3,4,5} = 1,1,1,0,-1$. The  eigenvectors associated to the three unstable eigenvalues $\tilde{\lambda}_{1,2,3}= 1,1,1$ are
\begin{align*}
\tilde{V}_1 = \left(1,0,0,0,0\right)^\top \quad \mbox{and} \quad \tilde{V}_3 = (0,0,1,1,1)^\top,
\end{align*}
where one needs another vector to span the generalized eigenspace associated to the triple eigenvalue $1$ with the aforementioned nontrivial Jordan block. Calculating a generalized eigenvector $\tilde{V}_2$ amounts to solving $(\tilde{A}-1\cdot\textnormal{Id})^2 \tilde{V}_2 = 0$, which has one solution given by
\[
\tilde{V}_2 = (0,2{\theta},2,1,0)^\top.
\]
The generalized eigenvectors span the unstable generalized eigenspace for the steady state $p_{\theta}$, given by the set of equations
\begin{align*}
F^\prime &= - \frac{1}{2 {\theta}} x_2 + F - \theta, \\
F^{\prime\prime} &= - \frac{1}{\theta} x_2 + F - {\theta},
\end{align*}
which is the same as~\eqref{es_n_ne_2} for $n = 2$. Consequently, equation~\eqref{eq_f_n_ne_2} remains satisfied.
\end{proof}
Note that we are in this case \emph{not able} to directly determine the local expansion near the contact line from~\eqref{eq_f_n_ne_2} as we have a differential equation for $F$, and not an algebraic equation as in~\eqref{eq:asymp1}. However, in analogy to Lemma~\ref{lem:ntcw}, one may show that the analytic map for the unstable manifold is nontrivial.
\begin{lemma}
\label{lem:nontrivtg}
The function $\tilde{g}$ is non-trivial. In particular, we have
\begin{equation}\label{der_n_ne_2}
\frac{\partial \tilde{g}_1}{\partial x_1}(0,0,{\theta})= 0, \qquad \frac{\partial \tilde{g}_1}{\partial x_2}(0,0,{\theta})= - \frac{{\theta}^{1-n}}{(n-3)(n-4)}, \qquad \frac{\partial \tilde{g}_1}{\partial F}(0,0,{\theta})=1.
\nonumber
\end{equation}
\end{lemma}

\subsection{Nonzero dynamic contact angles: resonances} \label{ssec:pw1}
To describe the actual trajectories we are interested in, it remains to study the non-autonomous ODE~\eqref{eq_f_n_ne_2} for $F=F(x)$. In particular, we would like to know when~\eqref{eq_f_n_ne_2} has analytic solutions, and if so, in which variables we have to consider the expansion. To explain dynamically why we expect the values $3-\frac1m$ for $m\in\N$ to be special, it is helpful to consider the full three-dimensional flow on the unstable manifold $W^\txtu(p_{\theta})$ given by
\begin{equation}
\label{3Dunstable}
\frac{\d}{\d s} \begin{pmatrix} x_1 \\ x_2 \\ u  \end{pmatrix}= \begin{pmatrix} x_1 \\  (3-n)x_2 \\ 
\tilde{g}_1(x_1,x_2,u + {\theta}) \end{pmatrix}=:\mathcal{A}\begin{pmatrix} x_1 \\ x_2 \\ u\end{pmatrix}
+\cN(x_1,x_2,u),
\end{equation}
where $u := F - {\theta}$ for some fixed $\theta>0$, $\mathcal{A}$ denotes the linear part, and $\cN(x_1,x_2,u)$ the nonlinear part. Regarding analyticity, one must ask whether~\eqref{3Dunstable} has locally convergent power series expansions as solutions near the origin. The question of convergence of formal power series and analytic equivalence of ODEs is a general problem~\cite{ArnoldGeomODE,Bibikov}. Let $q\in(\N_0)^3=(\N\cup\{0\})^3$, define $|q|:=\sum_{k=1}^3q_k$, and denote by
\[
\tilde \Lambda := \left(\tilde \lambda_1, \tilde \lambda_2, \tilde \lambda_3\right)^\top = \left(1, 3-n, 1\right)^\top
\]
the column vector containing the eigenvalues of $\mathcal{A}$. If
\begin{equation}
\label{eq:nonresonance}
q^\top \tilde \Lambda - \tilde \lambda_k\neq 0 \quad \mbox{for all} \quad k \in \{1,2,3\} \quad \mbox{and} \quad q \in \left(\N_0\right)^3 \quad \mbox{such that} \quad \verti{q} \ge 2,
\end{equation}
then the system~\eqref{3Dunstable} is called \emph{non-resonant}. If there exist $q$ and an index $k$ such that $q^\top \tilde \Lambda= \tilde \lambda_k$, then we have a \emph{resonance}~\cite{Bibikov}. It is well-known that resonances are connected to the failure of analytic equivalence and the non-removability of nonlinear polynomial factors with multi-index exponent $q$; see for example the assumptions in \cite[Thm.~3.2]{Bibikov}. For our case, the resonances can be computed.
\begin{lemma}\label{lem:reson}
If $n\in(0,3)$ then~\eqref{3Dunstable} is resonant in the following situations:
\begin{itemize}
 \item[(R1)] For $n=3-\frac1m$ with $m\in\N$, resonances occur with $q=(0,q_2,0)^\top$ for $q_2\geq 2$.
 \item[(R2)] If $n=2$, a resonance also occurs for $q \in \left\{(2,0,0)^\top, (1,0,1)^\top,(0,0,2)^\top\right\}$.
\end{itemize} 
\end{lemma}
\begin{proof}
The existence of resonances can be computed from the eigenvalues. To find resonances we have to solve (cf.~\eqref{eq:nonresonance})
\begin{equation}
\label{eq:reson}
q_1 + (3-n)q_2 + q_3 - \tilde\lambda_k=0,\quad n\in(0,3),\quad k\in\{1,2,3\}, \quad q\in(\N_0)^3, \quad \verti{q} \ge 2.
\end{equation}
Basically, \eqref{eq:reson} are three linear Diophantine equations with constraints, or alternatively viewed, constrained integer programming problems. Here we can solve the problem explicitly going through the three cases $k\in\{1,2,3\}$.

\medskip

We start with $k \in \{1,3\}$, so resonances occur if and only if
\begin{equation}\label{eq:reseq1}
q_1 + (3-n)q_2 + q_3 = 1 \quad \mbox{for} \quad \verti{q} \ge 2.
\end{equation}
If $q_2 = 0$ this implies $q_1 + q_3 = 1$ with $q_1 + q_3 \ge 2$ and no resonances occur. So suppose $q_2 \ne 0$: then one finds from~\eqref{eq:reseq1} that
\begin{equation}\label{eq:nsolve1}
n=\frac{q_1 + 3 q_2 + q_3 - 1}{q_2}.
\end{equation}
From the constraint $n\in(0,3)$ it follows that
\[
0 < q_1 + 3 q_2 + q_3-1< 3q_2 \quad \Rightarrow \quad q_1 + q_3< 1.
\]
Therefore, we must have $q_1=0=q_3$ and $q_2\geq 2$. From~\eqref{eq:nsolve1} it follows that $n=\frac{3 q_2 -1}{q_2}=3-\frac{1}{q_2}$ with $q_2\ge 2$. This gives all the resonances claimed in (R1) except for the one at $n=2$, and no other resonances for $k \in \{1,3\}$.

\medskip

Next, we focus on the case $k=2$, that is, resonances are present if
\begin{equation}\label{eq:reseq2}
q_1 + (3-n)q_2 + q_3 = 3-n \quad \mbox{for} \quad \verti{q} \ge 2.
\end{equation}
For $q_2 = 1$, \eqref{eq:reseq2} yields $q_1 + q_3 = 0$ and $q_1 + q_3 \ge 2$, so that no resonances occur. Otherwise, we may solve for $n$ and obtain
\[
n = \frac{q_1 + 3 q_2 + q_3 - 3}{q_2 - 1}.
\]
From the constraint $n \in (0,3)$ we conclude
\begin{equation}\label{eq:nsolve2}
0 < \frac{q_1 + 3 q_2 + q_3 - 3}{q_2 - 1} < 3.
\end{equation}
For $q_2 = 0$ we have $n = 3 - q_1 - q_2$, which leads to the resonance at $n = 2$ proving (R2). For $q_2 \ge 2$, the denominator in \eqref{eq:nsolve2} is positive and we obtain the constraints
\[
0 < q_1 + 3 q_2 + q_3 - 3 < 3 q_2 - 3.
\]
In particular $q_1 + q_3 < 0$ has to hold and therefore no further resonances can occur for $k = 2$.
\end{proof}

Lemma~\ref{lem:reson}, in addition to the discussion for $n=2$ resulting in~\eqref{eq:linsolven2}, clearly indicates that the values $n=3-\frac1m$ for $m\in\N$ should be special. In Section~\ref{ssec:pw2}, we are going to give a rigorous proof using a fixed-point argument for the existence and asymptotic expansion in the resonant and non-resonant cases. Here we briefly comment on the case (R1).

\medskip

The Poincar\'e-Dulac Theorem~\cite{ArnoldGeomODE} provides a way to \emph{formally conjugate} a system with resonances to a normal form. Condition (R1) leads to a normal form
\begin{equation}\label{formalnform}
\frac{\d}{\d s} \begin{pmatrix} \tilde x_1 \\ \tilde x_2 \\ w \end{pmatrix}= \begin{pmatrix} m \tilde x_1  \\  \tilde x_2 + mc_1 \tilde x_2^{q_2}\\ mw
\end{pmatrix}
\end{equation}
for some constant $c_1\in\R$ upon using a coordinate change $\left(\tilde x_1,\tilde x_2,w\right)=\Upsilon(x_1,x_2,u)$ and a time-rescaling $s\mapsto ms$. However, one may only express $\Upsilon$ via a \emph{formal series expansion} and one has to prove whether the series converges yielding an analytic change of coordinates, or whether the series even provides a topological equivalence of the original and transformed vector fields~\cite{Bibikov}. If the coordinate change would be analytic, one may simply transform~\eqref{formalnform} back to original coordinates, which yields the equation
\begin{equation}
\label{eq:resbad}
\frac{\txtd u}{\txtd s}=mu+\cG(x_1,x_2,u)
\end{equation}
where $\cG(x_1,x_2,u)$ is of order two or higher in $u$. Converting back to the $x$-coordinate via $s=\log x$, one ends up with a lowest-order system of the form
\begin{equation}
\label{eq:resbad1}
x\frac{\txtd u}{\txtd x}-mu=\cG(x,x^{1/m},u)
\end{equation}
which does generically lead to logarithmic terms in the solution for $u$; we are going to pick this observation up again in equation~\eqref{eq_u_res}. Although the argument is very appealing to obtain the expansion results near resonances, the problem remains that the Poincar\'e-Dulac Theorem does \emph{not} provide an analytic conjugacy in general~\cite{Bruno}. Therefore, and for the sake of a self-contained presentation that captures the dependence on parameters, we are going to use a fixed-point technique designed for the particular problem.

\subsection{Nonzero dynamic contact angles: fixed-point problem}
\label{ssec:pw2}
Recall from Lemma~\ref{lem:umfldequat} that it remains to study a one-dimensional non-autonomous ODE. Using
\begin{equation}\label{u_theta_0}
u=F - {\theta}
\end{equation}
it follows from~\eqref{eq_f_n_ne_2} that we have
\begin{equation}\label{eq_u_n_ne_2}
\left(x \frac{\d}{\d x} - 1\right) u = G\left(x_1,x_2,u\right) \quad \mbox{where} \quad x_1 = x, \quad x_2 = x^{3-n},
\end{equation}
and $G(x_1,x_2,u) := \tilde{g}_1(x_1,x_2,u+{\theta}) - u$. In particular $G(0,0,0) = \frac{\partial G}{\partial u}(0,0,0) = 0$ (cf.~\eqref{der_n_ne_2}). For $n \notin \left\{3 - \frac 1 m: \, m = 1, 2, 3, \cdots\right\}$ (following the strategy in~\cite[\S 2]{ggo.2013}) we may study~\eqref{eq_u_n_ne_2} by treating $x_1$ and $x_2$ as independent variables, replacing $u = u(x)$ by $v = v(x_1,x_2)$ such that
\begin{equation}\label{identify}
v\left(x, x^{3-n}\right) = u(x),
\end{equation}
and substituting the operator $x \frac{\d}{\d x}$ by $x_1 \partial_{x_1} + (3-n) x_2 \partial_{x_2}$, since indeed
\[
\left(x_1 \partial_{x_1} + (3-n) x_2 \partial_{x_2}\right) v\left(x_1, x_2\right) = x \frac{\d u}{\d x}(x) \quad \mbox{if} \quad  \left(x_1,x_2\right) = \left(x, x^{3-n}\right).
\]
The resulting problem reads
\begin{equation}\label{eq_v_n_ne_2}
\left(x_1 \partial_{x_1} + (3-n) x_2 \partial_{x_2} - 1\right) v = G\left(x_1,x_2,v\right) \quad \mbox{around} \quad (x_1,x_2) = (0,0).
\end{equation}
Then through~\eqref{identify}, constructing a solution to~\eqref{eq_v_n_ne_2} automatically yields a solution to~\eqref{eq_u_n_ne_2}. Observe that the regular function $x_1$ is in the kernel of the linear operator in \eqref{eq_v_n_ne_2}. Hence, we may impose the boundary conditions
\begin{equation}\label{bc_nonres}
\left(v, \partial_{x_1} v\right) = (0,b) \quad \mbox{at} \quad (x_1,x_2) = (0,0),
\end{equation}
where $b$ determines the curvature of the solution $H$ to problem~\eqref{ode_third2}\&\eqref{bcx1} and will be used as a shooting parameter in Section~\ref{sec:global} to match condition~\eqref{bcx2}. Since all characteristics of the differential operator $x_1 \partial_{x_1} + (3-n) x_2 \partial_{x_2} - 1$ meet in the origin $(x_1, x_2) = (0,0)$, the boundary conditions \eqref{bc_nonres} are sufficient in order to obtain existence and uniqueness to problem~\eqref{eq_v_n_ne_2}-\eqref{bc_nonres}. However, the dependence of $v$ on the parameter $b$ is implicit, and in order to make it \emph{explicit}, we may further unfold according to
\begin{equation}\label{identify_2}
w(x_1,x_2,x_3) + x_3 = v(x_1,x_2) \quad \mbox{with} \quad x_3 = b x_1,
\end{equation}
so that
\[
\left(x_1 \partial_{x_1} + (3-n) x_2 \partial_{x_2} + x_3 \partial_{x_3}\right) w + x_3 = \left(x_1 \partial_{x_1} + (3-n) x_2 \partial_{x_2}\right) v  \quad \mbox{if} \quad x_3 = b x_1.
\]
The resulting problem is
\begin{subequations}\label{prob_w_nonres}
\begin{align}
\left(x_1 \partial_{x_1} + (3-n) x_2 \partial_{x_2} + x_3 \partial_{x_3} - 1\right) w &= G\left(x_1,x_2,w+x_3\right) \quad \mbox{around} \quad (x_1,x_2,x_3) = (0,0,0), \label{eq_w_nonres}\\
\left(w,\partial_{x_1} w, \partial_{x_3} w\right) &= (0,0,0) \quad \mbox{at} \quad (x_1,x_2,x_3) = (0,0,0). \label{bc_w_nonres}
\end{align}
\end{subequations}
Indeed, as $b$ does \emph{not} appear in \eqref{prob_w_nonres} anymore, $w$ is \emph{independent} of $b$.
\begin{proposition}[non-resonant case]\label{prop:nonres}
Suppose $n \in (0,3)$ and $n\neq 3-\frac1m$ for $m\in\N$. Then for $0 < \eps \ll 1$ problem~\eqref{prob_w_nonres} has an analytic solution $w=w(x_1,x_2,x_3)$ for $(x_1,x_2,x_3)\in[0,\eps] \times [0,\eps^2] \times [-\eps,\eps]$.
\end{proposition}
\begin{proof}
 For $n \notin \left\{3 - \frac 1 m: \, m \in \N\right\}$, we may recast~\eqref{prob_w_nonres} in form of the fixed-point equation
\begin{equation}\label{fixed_v}
w = \SSS[w] := \TT G\left(x_1,x_2,w+x_3\right)
\end{equation}
using analyticity of $G$, where $x_1$ is in the kernel of the linear operator in \eqref{eq_w_nonres} and
\begin{equation}\label{def_t}
\begin{aligned}
\TT g (x_1,x_2,x_3) := \sum_{(k,\ell,p) \in \II} \frac{\partial_{x_1}^k \partial_{x_2}^\ell \partial_{x_3}^p g(0,0,0)}{k + (3-n) \ell + p - 1} x_1^k x_2^\ell x_3^p \\
\mbox{with} \quad \II = \left(\N_0\right)^3 \setminus \left\{(0,0,0), (1,0,0), (0,0,1)\right\}.
\end{aligned}
\end{equation}
Note that the denominator in the series in~\eqref{def_t} is nonzero in the \emph{non-resonant case} (i.e., for $n \in (0,3)$ if $n\neq 3-\frac1m$ with $m\in \N$). Also note that $G(x_1,x_2,x_3+w)$ (and therefore also $g$) fulfills conditions~\eqref{bc_w_nonres} because of the definition of $G$ after \eqref{eq_u_n_ne_2}, Lemma~\ref{lem:nontrivtg}, and $\partial_u G(0,0,0) = 0$. In order to construct a solution to \eqref{fixed_v}, we introduce 
\begin{equation}\label{def_norm}
\vertii{w} := \sum_{(k,\ell,p) \in \II} \frac{\eps^{k + 2 \ell +p}}{k ! \ell ! p !} \verti{\partial_{x_1}^k \partial_{x_2}^\ell \partial_{x_3}^p w(0,0,0)} \quad \mbox{with} \quad \eps > 0,
\end{equation}
which is a sub-multiplicative norm on the space of all analytic $w$ with $\vertii{w} < \infty$ meeting conditions~\eqref{bc_w_nonres}.

\medskip

We first show that the map $\SSS$ is a self-map in $\left\{w: \, \vertii{w} \le \delta\right\}$ for $0 < \eps \ll \delta \ll 1$. Indeed, the linear operator $\TT$ (cf.~\eqref{def_t}) allows for the estimate $\vertii{\TT g} \lesssim \vertii{g}$, i.e., there exists a constant $C >0$ (depending only upon $n$) such that $\vertii{\TT g} \leq C \vertii{g}$ holds uniformly in the small parameters. Therefore, we may conclude that
\begin{equation}
\label{eq:fixedest1}
\vertii{\SSS[w]} \stackrel{\eqref{fixed_v}}{\lesssim} \vertii{G\left(x_1,x_2,w+x_3\right)} \lesssim \eps^2 + \delta^2 \quad \mbox{for} \quad \vertii{w} \le \delta \quad \mbox{and} \quad \eps, \delta \le 1;
\end{equation}
observe that the second inequality in~\eqref{eq:fixedest1} follows just looking at the leading-order terms of $G$ and employing sub-multiplicativity of $\vertii{\cdot}$. This shows our claim for $0 < \eps \ll \delta  \ll 1$.

\medskip

By the same reasoning also
\begin{eqnarray*}
\vertii{\SSS\left[w^{(1)}\right] - \SSS\left[w^{(2)}\right]} &\stackrel{\eqref{fixed_v}}{\lesssim}& \vertii{G\left(x_1,x_2,w^{(1)}+x_3\right)-G\left(x_1,x_2,w^{(2)}+x_3\right)} \\
&\lesssim& (\eps + \delta) \vertii{w^{(1)} - w^{(2)}}
\end{eqnarray*}
for $\vertii{w^{(j)}} \le \delta$ and $\eps, \delta \le 1$. Then the contraction property follows for $0 <  \eps, \delta \ll 1$.

\medskip

The contraction-mapping theorem yields existence of a solution $w$ to~\eqref{fixed_v} with finite norm $\vertii{w}$, hence an analytic solution $w = w(x_1,x_2,x_3)$ for $(x_1,x_2,x_3) \in [0,\eps] \times [0,\eps^2] \times [-\eps,\eps]$.
\end{proof}
Next, we turn our attention to the resonant case for which $(3-n)^{-1} =: m = 1, 2, 3, \cdots$. Then equation~\eqref{eq_u_n_ne_2} changes to
\begin{equation}\label{eq_u_res}
\left(y \frac{\d}{\d y} - m\right) u = G\left(y,u\right) \quad \mbox{where} \quad y = x^{\frac 1 m}
\end{equation}
and $G(y,u) := m \left(\tilde{g}_1\left(y^m,y,u+{\theta}\right) - u\right)$. Equation~\eqref{eq_u_res} only \emph{seems} to be structurally simpler than~\eqref{eq_u_n_ne_2}, as logarithmic terms may appear. This can be easily seen by studying the corresponding linear problem
\[
\left(y \frac{\d}{\d y} - m\right) u = g(y) \quad \mbox{where} \quad y = x^{\frac 1 m}.
\]
Assuming the simple and generic case $g(y) = y^m$, we obtain solutions of the form $u(y) = y^m \log y + C y^m$ with a constant $C \in \R$. Hence it appears natural to replace $u = u(y)$ by $v = v(y_1,y_2)$ with
\begin{equation}\label{identify_y}
v(y_1,y_2) = u(y) \quad \mbox{if} \quad (y_1,y_2) = \left(y^m \log y,y\right),
\end{equation}
and to substitute the operator $y \frac{\d}{\d y}$ by $\left(m y_1 + y_2^m\right) \partial_{y_1} + y_2 \partial_{y_2}$, so that
\[
\left(\left(m y_1 + y_2^m\right) \partial_{y_1} + y_2 \partial_{y_2}\right) v\left(y_1, y_2\right) = y \frac{\d u}{\d y} (y) \quad \mbox{if} \quad  \left(y_1,y_2\right) = \left(y^m \log y, y\right).
\]
Instead of~\eqref{eq_u_res} we may solve
\begin{equation}\label{eq_u_res_un}
\left(\left(m y_1 + y_2^m\right) \partial_{y_1} + y_2 \partial_{y_2} - m\right) v = G(y_2,v) \quad \mbox{around} \quad (y_1,y_2) = (0,0).
\end{equation}
Notice that the monomial $y_2^m$ is in the kernel of the linear operator in~\eqref{eq_u_res_un} and we will indeed construct solutions meeting the boundary conditions
\begin{equation}\label{bc_res}
\left(v, \partial_{y_2}^m v\right) = \left(0, b \, m !\right) \quad \mbox{at} \quad (y_1,y_2) = (0,0).
\end{equation}
Again, $b$ determines the curvature of the solution $H$ to problem~\eqref{ode_third2}\&\eqref{bcx1} and will be used as a shooting parameter in Section~\ref{sec:global} to meet condition~\eqref{bcx2}.  Also note that only boundary conditions in the single point $(y_1,y_2) = (0,0)$ have to be assumed as all characteristics of the linear operator $\left(m y_1 + y_2^m\right) \partial_{y_1} + y_2 \partial_{y_2}$ end in (emanate from) this single point.

\medskip

As a last step, we make the dependence on $b$ explicit, by identifying
\begin{equation}\label{identify_y_2}
w(y_1,y_2,y_3) + y_3 = v(y_1,y_2) \quad \mbox{if} \quad y_3 = b y_2^m,
\end{equation}
so that
\[
\left(\left(m y_1 + y_2^m\right) \partial_{y_1} + y_2 \partial_{y_2} + m y_3 \partial_{y_3}\right) w + m y_3 = \left(\left(m y_1 + y_2^m\right) \partial_{y_1} + y_2 \partial_{y_2}\right) v \quad \mbox{if} \quad x_3 = b y_2^m.
\]
The resulting problem has the following form:
\begin{subequations}\label{prob_w_res}
\begin{align}
\left(\left(m y_1 + y_2^m\right) \partial_{y_1} + y_2 \partial_{y_2} + m y_3 \partial_{y_3} - m\right) w &= G\left(y_2,w+y_3\right) \quad \mbox{around} \quad (y_1,y_2,y_3) = (0,0,0),\\
\left(w,\partial_{y_2}^m w, \partial_{y_3} w\right) &= (0,0,0) \quad \mbox{at} \quad (y_1,y_2,y_3) = (0,0,0). \label{bc_w_res}
\end{align}
\end{subequations}
%
\begin{proposition}[resonant case]\label{prop:res}
Suppose $n \in (0,3)$ and $n=3-\frac1m$ for $m\in\N$. Then for $0 < \eps \ll 1$, problem~\eqref{prob_w_res} has an analytic solution $w=w(y_1,y_2,y_3)$ for $(y_1,y_2,y_3) \in [0,\eps^2] \times [0,\eps^2] \times [-\eps,\eps]$. 
\end{proposition}
\begin{proof}
Note that~\eqref{eq_u_res_un} can be converted into the fixed-point problem
\begin{equation}\label{fixed_v_res}
w = \SSS[w] := \TT G\left(y_1,w + y_3\right),
\end{equation}
where $\TT$ is the uniquely defined linear solution operator to
\begin{subequations}\label{lin_res}
\begin{align}
\left(\left(m y_1 + y_2^m\right) \partial_{y_1} + y_2 \partial_{y_2} + m y_3 \partial_{y_3} - m\right) \TT g &= g,\\
\left(g,\partial_{y_1} g, \partial_{y_3} g, \TT g, \partial_{y_2}^m \TT g, \partial_{y_3} \TT g\right) &= (0,0,0,0,0,0) \quad \mbox{at} \quad (y_1,y_2,y_3) = (0,0,0).
\end{align}
\end{subequations}
The second and third boundary condition $\partial_{y_1} g(0,0,0) = 0$ and $\partial_{y_3} g(0,0,0) = 0$ can be assumed as for $g$ replaced by $G(y_2,v) = G(y_2,w + y_3)$ an explicit dependence in $y_1$ does \emph{not} occur, $\partial_v G(0,0) = 0$, and $w(0,0,0) = 0$. We may use the power series representation
\begin{align*}
g(y_1,y_2,y_3) &= \sum_{(k,\ell,p) \in \left(\N_0\right)^3} \frac{\partial_{y_1}^k \partial_{y_2}^\ell \partial_{y_3}^p g(0,0,0)}{k! \ell ! p !} y_1^k y_2^\ell y_3^p, \\
\TT g(y_1,y_2,y_3) &= \sum_{(k,\ell,p) \in \left(\N_0\right)^3} \frac{\partial_{y_1}^k \partial_{y_2}^\ell \partial_{y_3}^p \TT g(0,0,0)}{k! \ell ! p !} y_1^k y_2^\ell y_3^p,
\end{align*}
where
\begin{align*}
\partial_{y_1}^k \partial_{y_2}^\ell \partial_{y_3}^p g(0,0,0) &= 0 \quad \mbox{for} \quad (k,\ell,p) \in \left(\N_0\right)^3 \setminus \II, \\
\mbox{with} \quad \II &:= \left(\N_0\right)^3 \setminus \left\{(0,0,0),(1,0,0),(0,0,1)\right\}, \\
\partial_{y_1}^k \partial_{y_2}^\ell \partial_{y_3}^p \TT g(0,0,0) &= 0 \quad \mbox{for} \quad (k,\ell,p) \in \left(\N_0\right)^3 \setminus \JJ,\\
\mbox{with} \quad \JJ &:= \left(\N_0\right)^3 \setminus \left\{(0,0,0),(0,m,0),(0,0,1)\right\},
\end{align*}
leading to
\begin{eqnarray*}
\lefteqn{\left(\left(m y_1 + y_2^m\right) \partial_{y_1} + y_2 \partial_{y_2} + m y_3 \partial_{y_3} - m\right) \TT g (y_1,y_2)} \\
&=& \sum_{(k,\ell,p) \in \left(\N_0\right)^3} \frac{(m (k +p-1)+ \ell) \partial_{y_1}^k \partial_{y_2}^\ell \partial_{y_3}^p \TT g(0,0,0)}{k! \ell ! p !} y_1^k y_2^\ell y_3^p \\
&& + \sum_{\ell \ge m} \frac{\ell \cdots (\ell-m+1) \partial_{y_1}^{k+1} \partial_{y_2}^{\ell-m} \partial_{y_3}^p \TT g(0,0,0)}{k ! \ell ! p !} y_1^k y_2^\ell y_3^p.
\end{eqnarray*}
Inserted into~\eqref{lin_res}, a comparison of coefficients leads to
\begin{subequations}\label{comp_coeff}
\begin{equation}
(m (k+p-1) + \ell) \partial_{y_1}^k \partial_{y_2}^\ell \partial_{y_3}^p \TT g(0,0,0) = \partial_{y_1}^k \partial_{y_2}^\ell \partial_{y_3}^p g(0,0,0)
\end{equation}
for $\ell < m$ and
\begin{equation}
\begin{aligned}
& (m (k+p-1) + \ell) \partial_{y_1}^k \partial_{y_2}^\ell \partial_{y_3}^p \TT g(0,0,0) + \ell \cdots (\ell-m+1) \partial_{y_1}^{k+1} \partial_{y_2}^{\ell-m} \partial_{y_3}^p \TT g(0,0,0) \\
&\quad = \partial_{y_1}^k \partial_{y_2}^\ell \partial_{y_3}^p g(0,0,0)
\end{aligned}
\end{equation}
\end{subequations}
for $\ell \ge m$. Now we may define the operator $\TT$ using~\eqref{comp_coeff} by setting
\begin{subequations}\label{def_tb}
\begin{align}\label{def_t_1}
\partial_{y_1}^k \partial_{y_2}^\ell \partial_{y_3}^p \TT g(0,0,0) &:= \frac{\partial_{y_1}^k \partial_{y_2}^\ell \partial_{y_3}^p g(0,0,0)}{m (k+p-1) + \ell} \quad \mbox{for} \quad (k,\ell,p) \in \II \quad \mbox{with} \quad \ell < m \\
\partial_{y_1} \TT g(0,0,0) &:= \frac{\partial_{y_2}^m g(0,0,0)}{m!},\label{def_t_2}
\end{align}
and then defining inductively in $\ell$
\begin{equation}\label{def_t_3}
\partial_{y_1}^k \partial_{y_2}^\ell \partial_{y_3}^p \TT g(0,0,0) := \frac{\partial_{y_1}^k \partial_{y_2}^\ell \partial_{y_3}^p g(0,0,0) - \ell \cdots (\ell-m+1) 
\partial_{y_1}^{k+1} \partial_{y_2}^{\ell-m} \partial_{y_3}^p \TT g(0,0,0)}{m (k+p-1) + \ell}
\end{equation}
\end{subequations}
for $\ell \ge m$ with $(k,\ell,p) \ne (0,m,0)$. This procedure uniquely determines the operator $\TT$ for functions $g = g(y_1,y_2,y_3)$ that are analytic in a neighborhood of the origin $(y_1,y_2,y_3) = (0,0)$ with $g(0,0,0) = \partial_{y_{1}} g(0,0,0) = \partial_{y_3} g(0,0,0) = 0$. We can also prove estimates for $\TT$ (with constants only depending on $n$): From \eqref{def_t_1} and \eqref{def_t_2} it is immediate that
\begin{subequations}\label{est_t_0}
\begin{equation}
\frac{\verti{\partial_{y_1}^k \partial_{y_2}^\ell \partial_{y_3}^p \TT g(0,0,0)}}{k! \ell ! p !} \lesssim \frac{\verti{\partial_{y_1}^k \partial_{y_2}^\ell \partial_{y_3}^p g(0,0,0)}}{k ! \ell ! p !} \quad \mbox{for} \quad (k,\ell,p) \in \II \quad \mbox{with} \quad \ell < m.
\end{equation}
From~\eqref{def_t_2} trivially
\begin{equation}
\verti{\partial_{y_1} \TT g(0,0,0)} \lesssim \frac{\verti{\partial_{y_2}^m g(0,0,0)}}{m!}.
\end{equation}
\end{subequations}
For $\ell \ge m$ and $(k,\ell,p) \ne (0,m,0)$, we obtain from~\eqref{def_tb} and~\eqref{est_t_0} by induction
\begin{subequations}\label{est_t_1}
\begin{equation}
\begin{aligned}
\frac{\verti{\partial_{y_1}^k \partial_{y_2}^\ell \partial_{y_3}^p \TT g(0,0,0)}}{k! \ell ! p !} &\le \sum_{m j \le \ell} \frac{\prod_{i = 1}^j (k+i)}{\verti{m (k-1) + \ell}^{j+1}} \frac{\verti{\partial_{y_1}^{k+j} \partial_{y_2}^{\ell-m j} \partial_{y_3}^p g(0,0,0)}}{(k+j) ! (\ell-mj) ! p !} \\
& \lesssim \frac 1 \ell \sum_{m j \le \ell} \frac{\verti{\partial_{y_1}^{k+j} \partial_{y_2}^{\ell-mj} \partial_{y_3}^p g(0,0,0)}}{(k+j) ! (\ell-mj) ! p !},
\end{aligned}
\end{equation}
where $k \ne 0$ or $\ell \notin m \N_1$ or $p \ne 0$ (note that $g(0,0,0) = 0$). In the particular case $k = 0$ and $\ell \in m \N_1$ and $p = 0$, an analogous computation shows 
\begin{equation}
\frac{\verti{\partial_{y_2}^\ell \TT g(0,0,0)}}{\ell !} \lesssim \frac 1 \ell \sum_{j = 0}^{\frac \ell m - 1} \frac{\verti{\partial_{y_1}^j \partial_{y_2}^{\ell-mj} g(0,0,0)}}{j ! (\ell-mj) !} + \frac{\verti{\partial_{y_2}^m g(0,0,0)}}{m!},
\end{equation}
\end{subequations}
Estimates~\eqref{est_t_0} and \eqref{est_t_1} imply
\begin{equation}\label{est_t_y}
\vertii{\TT g} \lesssim \vertii{g} \quad \mbox{with} \quad \vertii{w} := \sum_{(k,\ell,p) \in \JJ} \frac{\eps^{2 m k + 2 \ell + p}}{k! \ell! p!} \verti{\partial_{y_1}^k \partial_{y_2}^\ell \partial_{y_3}^p w(0,0,0)}.
\end{equation}
Using the linear estimate \eqref{est_t_y}, the contraction can be verified in almost exactly the same way as in the non-resonant case and yields existence of an analytic solution $w = w(y_1,y_2,y_3)$ of \eqref{fixed_v_res} in the cuboid $(y_1, y_2,y_3) \in [0,\eps^2] \times [0,\eps^2] \times [-\eps,\eps]$.
\end{proof}
In the following corollary, we discuss the dependence on the parameter $b$:
\begin{corollary}\label{cor:exp_hb}
Let $H$ denote the solution to problem~\eqref{ode_third2}\&\eqref{bcx1} with $\theta > 0$. Then
\begin{equation}\label{as_hdhd2h_theta}
H\left(\eps \verti{b}^{-1}\right) \sim \eps \verti{b}^{-1}, \qquad \frac{\d H}{\d x}\left(\eps \verti{b}^{-1}\right) \sim 1, \qquad \mbox{and} \qquad \frac{\d^2 H}{\d x^2}\left(\eps \verti{b}^{-1}\right) \sim b
\end{equation}
for $0 < \eps \ll 1$ and $\verti{b} \gg_\eps 1$, where $b$ determines $H$ through \eqref{bc_nonres} or \eqref{bc_res}, respectively.
\end{corollary}
\begin{proof}[Proof of Corollary~\ref{cor:exp_hb} (non-resonant case)]
For $n \in (0,3) \setminus \{3-\frac1m: \, m \in \N\}$, we have due to equations \eqref{trafo_pw_0n3}\&\eqref{u_theta_0}\&\eqref{identify}\&\eqref{identify_2}
\begin{equation}\label{rep_d2h_nonres}
H = x \left(\theta + b x + w\left(x,x^{3-n},bx\right)\right) \quad \mbox{for} \quad 0 \le x \le c \min\left\{\eps,\eps^{\frac{2}{3-n}}, \frac{\eps}{\verti{b}}\right\} =: x_b^*(\eps),
\end{equation}
with a sufficiently small constant $c > 0$, and $\eps > 0$ and the analytic solution $w = w(x_1,x_2,x_3)$ to \eqref{prob_w_nonres} are chosen as in or given by Proposition~\ref{prop:nonres}, respectively. By absorbing into $\eps$, we may assume without loss of generality $c = 1$. Differentiating \eqref{rep_d2h_nonres} with respect to $x$, we obtain
\begin{align*}
\frac{\d H}{\d x} &= \theta + 2 b x + \left(x \frac{\d}{\d x} + 1\right) w = \theta + 2 b x + \left(x_1 \partial_{x_1} + (3-n) x_2 \partial_{x_2} + x_3 \partial_{x_3}+1\right) w,\\
\frac{\d^2 H}{\d x^2} &= 2 b + x^{-1} x \frac{\d}{\d x} \left(x \frac{\d}{\d x} + 1\right) w = 2 b + x^{-1} q\left(x_1 \partial_{x_1} + (3-n) x_2 \partial_{x_2} + x_3 \partial_{x_3}\right) w,
\end{align*}
where $q(\zeta) := \zeta(\zeta+1)$ and $(x_1,x_2,x_3) = \left(x,x^{3-n},bx\right)$. Due to the boundary conditions \eqref{bc_w_nonres}, we obtain
\begin{equation}\label{as_d2h_nonres}
\left.
\begin{aligned}
H &= \theta x + b x^2 + \cO\left(x^3 + x^{4-n} + \verti{b}^2 x^3\right) \\
\frac{\d H}{\d x} &= \theta + 2 b x + \cO\left(x^2 + x^{3-n} + \verti{b}^2 x^2\right) \\
\frac{\d^2 H}{\d x^2} &= 2 b + \cO\left(x + x^{2-n} + \verti{b}^2 x\right)
\end{aligned}
\right\} \quad \mbox{for} \quad 0 \le x \le x_b^*(\eps).
\end{equation}
Evaluating \eqref{as_d2h_nonres} at $x = x_b^*(\eps)$, we get the limiting behavior
\[
\left.
\begin{aligned}
H\left(\eps \verti{b}^{-1}\right) &= \theta \eps \verti{b}^{-1} \left(1 + \cO\left(\eps^2 \verti{b}^{-2} + \eps^{3-n} \verti{b}^{n-3} + \eps\right)\right) \\
\frac{\d H}{\d x}\left(\eps \verti{b}^{-1}\right) &= \theta + \cO\left(\eps^2 \verti{b}^{-2} + \eps^{3-n} \verti{b}^{n-3} + \eps\right) \\
\frac{\d^2 H}{\d x^2}\left(\eps \verti{b}^{-1}\right) &= b \left(2 + \cO\left(\eps \verti{b}^{-2} + \eps^{2-n} \verti{b}^{n-3} + \eps \right)\right)
\end{aligned}
\right\} \quad \mbox{as} \quad b \to \pm \infty.
\]
Hence, for $0 < \eps \ll 1$ and $b \gg_\eps 1$, we arrive at \eqref{as_hdhd2h_theta}.
\end{proof}

\begin{proof}[Proof of Corollary~\ref{cor:exp_hb} (resonant case)]
In the resonant case $n = 3 - \frac 1 m$ with $m \in \N$, we use the sequence of transformations \eqref{trafo_pw_0n3}\&\eqref{u_theta_0}\&\eqref{identify_y}\&\eqref{identify_y_2} and obtain
\begin{equation}\label{rep_d2h_res}
H = x \left(\theta + b x + w\left(\frac 1 m x \log x,x^{\frac 1 m},bx\right)\right) \quad \mbox{for} \quad 0 \le x \le c \min\left\{\tilde c(\eps), \frac{\eps}{\verti{b}}\right\} =: x_b^{**}(\eps),
\end{equation}
where $c > 0$ is sufficiently small, $\tilde c(\eps) > 0$ is an $\eps$-dependent constant, and $\eps > 0$ and the analytic solution $w = w(x_1,x_2,x_3)$ to \eqref{prob_w_nonres} are chosen or constructed in Proposition~\ref{prop:nonres}, respectively. Again, we may assume $c = 1$. Differentiating in $x$ yields
\begin{align*}
\frac{\d H}{\d x} &= \theta + 2 b x + \left(x \frac{\d}{\d x} + 1\right) w \\
&= \theta + 2 b x + \left(m^{-1} \left(\left(m y_1 + y_2^m\right) \partial_{y_1} + y_2 \partial_{y_2} + m y_3 \partial_{y_3}\right) + 1\right) w,\\
\frac{\d^2 H}{\d x^2} &= 2 b + x^{-1} x \frac{\d}{\d x} \left(x \frac{\d}{\d x} + 1\right) w \\
&= 2 b + x^{-1} q\left(m^{-1} \left(\left(m y_1 + y_2^m\right) \partial_{y_1} + y_2 \partial_{y_2} + m y_3 \partial_{y_3}\right) \right) w,
\end{align*}
where $q(\zeta) := \zeta (\zeta+1)$ and $\left(y_1,y_2,y_3\right) = \left(\frac 1 m x \log x, x^{\frac 1m}, b x\right)$. Due to \eqref{bc_w_res}, we may conclude
\begin{equation}\label{as_d2h_res}
\left.
\begin{aligned}
H &= \theta x + b x^2 + \cO\left(x^2 \verti{\log x} + x^{1 + \frac 1 m} + \verti{b}^2 x^3\right) \\
\frac{\d H}{\d x} &= \theta + 2 b x + \cO\left(x \verti{\log x} + x^{\frac 1 m} + \verti{b}^2 x^2\right) \\
\frac{\d^2 H}{\d x^2} &= 2 b + \cO\left(\verti{\log x} + x^{- 1 + \frac 1 m} + \verti{b}^2 x\right)
\end{aligned}
\right\}\quad \mbox{for} \quad 0 \le x \le x_b^{**}(\eps).
\end{equation}
We evaluate \eqref{as_d2h_res} at $x = x_b^{**}(\eps)$ and obtain
\[
\left.
\begin{aligned}
H\left(\eps \verti{b}^{-1}\right) &= \theta \eps \verti{b}^{-1} \left(1 + \cO\left(\eps \verti{b}^{-1} \verti{\log\left(\eps \verti{b}^{-1}\right)} + \eps^{\frac 1 m} \verti{b}^{-\frac 1 m} + \eps\right)\right) \\
\frac{\d H}{\d x}\left(\eps \verti{b}^{-1}\right) &= \theta + \cO\left(\eps \verti{b}^{-1} \verti{\log\left(\eps \verti{b}^{-1}\right)} + \eps^{\frac 1 m} \verti{b}^{-\frac 1 m} + \eps\right) \\
\frac{\d^2 H}{\d x^2}\left(\eps \verti{b}^{-1}\right) &= b \left(2 + \cO\left(\verti{b}^{-1} \verti{\log\left(\eps \verti{b}^{-1}\right)} + \eps^{- 1 + \frac 1 m} \verti{b}^{-\frac 1 m} + \eps\right)\right)
\end{aligned}
\right\} \quad \mbox{as} \quad b \to \pm \infty
\]
and because of $n < 3$, we again arrive at \eqref{as_hdhd2h_theta} for $0 < \eps \ll 1$ and $b \gg_\eps 1$.
\end{proof}

\section{Global solutions}
\label{sec:global}
In this section we study the global aspects of~\eqref{problem_x}. In particular, we are going to prove the existence and uniqueness of solutions, with local analytic series expansions near the contact line described in Section~\ref{sec:local}, by matching the remaining boundary condition \eqref{bcx2} at $x=1$. We emphasize that existence of solutions has been proven in \cite{bpw.1992} for zero contact angles $\theta = 0$, so that in principle only for $\theta > 0$ our arguments are necessary. However, since our method is quite different from the one in \cite{bpw.1992} (where solutions are constructed by shooting from $x = 1$ and matching conditions \eqref{bcx1}), we opt for a self-contained presentation also for $\theta = 0$.

\subsection{Existence (shooting argument) for zero and nonzero dynamic contact angles}
\label{sec:globalexist}

The proof is based upon a shooting argument. First, it is convenient to re-write~\eqref{problem_x} as a first-order autonomous system using $\frac{\txtd x}{\txtd \tau}=1$, $H^\prime := \frac{\txtd H}{\txtd \tau}$, $H^{\prime\prime} := \frac{\txtd^2 H}{\txtd \tau^2}$, so that
\begin{equation}
\label{eq:ODE_f}
\frac{\d}{\d \tau} \begin{pmatrix} x \\ H \\ H^\prime \\ H^{\prime\prime} \end{pmatrix} = \begin{pmatrix} 1 \\ H^\prime \\ H^{\prime\prime} \\ H^{1-n}(x-1) \end{pmatrix}
\end{equation}
Note that $H'$, $H''$ are now notations for phase space variables. To consider the problem from a geometric perspective, one introduces the boundary conditions as boundary manifolds
\be
\label{eq:ODE_BVP}
\begin{array}{rl}
\cB_0^\theta&:=\{\left(x, H, H^\prime, H^{\prime\prime}\right)\in\R^4:x=0,H=0,H^\prime=\theta\},\\
\cB_1&:=\{\left(x, H, H^\prime, H^{\prime\prime}\right)\in\R^4:x=1,H^\prime=0\}.\\
\end{array}
\ee
In particular, solutions to~\eqref{eq:ODE_f}-\eqref{eq:ODE_BVP} correspond to connecting orbits between $\cB_0^\theta$ and $\cB_1$ as shown in Figure~\ref{fig03}.
\begin{figure}[htbp]
\psfrag{B0}{$\cB_0^\theta$}
\psfrag{x}{$x$}
\psfrag{u}{$H'$}
\psfrag{v}{$H''$}
\psfrag{u0}{$\cB_1$}
\psfrag{ut}{$\{x=0\}$}
\psfrag{Bd}{$\{x=\delta\}$}
\psfrag{Bd1}{$\cB_\delta^\theta$}
\psfrag{B1}{$\{x=1\}$}
	\centering
		\includegraphics[width=0.5\textwidth]{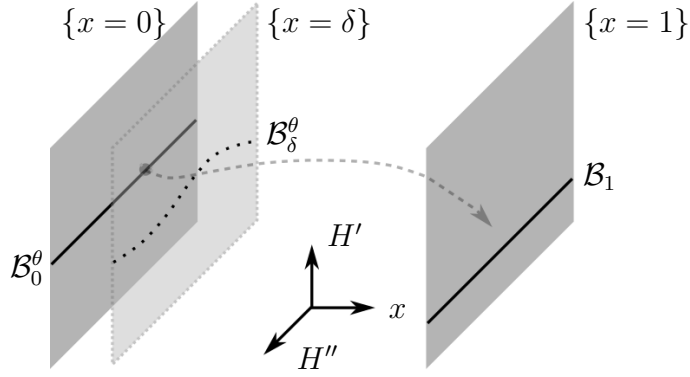}
		\caption{Sketch of the boundary value problem (BVP) geometric formulation. The starting sub-manifold $\cB_0^\theta$ is one-dimensional as it also accounts for $H=0$ while the target manifold $\cB_1$ is two-dimensional. A possible starting point (dot) in $\cB_0^\theta$ of a trajectory (dashed with an arrow at the end) is shown. This trajectory overshoots the condition $H^\prime=0$ at $x=1$.}
		\label{fig03}
\end{figure} 

Let $\phi_{\tau}\left(x, H, H^\prime, H^{\prime\prime}\right)$ denote the flow associated to the vector field~\eqref{eq:ODE_f}.
\begin{lemma}
\label{lem:shift}
There exists $\tau_0>0$ such that $\phi_{\tau}(\cB_0^\theta)$ is a well-defined analytic manifold for all $\tau \in (0,\tau_0]$ with 
$\dim(\{\phi_{\tau}(\cB_0^\theta):\tau\in(0,\tau_0]\})=\dim(\cB_0^\theta)+1$.
\end{lemma}
\begin{proof}
The result follows from the local analytic expansions of the solutions to the thin film equation near the contact line developed in Section~\ref{sec:local} (cf.~Proposition~\ref{prop:steady_0n32}, Proposition~\ref{prop:nonres}, and Proposition~\ref{prop:res}) if we select $\tau_0>0$ sufficiently small.
\end{proof}
Since $\frac{\d}{\d \tau} x = 1$, Lemma~\ref{lem:shift} implies that the definition
\be
\cB_\delta^\theta:=\{\phi_{\tau}(\cB_0^\theta):\tau\in[0,\tau_0]\} \cap \{x = \delta\}
\ee
yields a well-defined analytic sub-manifold for any $\delta\in(0,\tau_0]$. It is expected that connecting orbits are generically isolated since they correspond to intersections of $\{\phi_{\tau}(\cB_\delta^\theta):\tau\in[\delta,1]\}$ and $\{\phi_{-\tau}(\cB_1):\tau\in[0,1]\}$ and
\[
\dim(\{\phi_{\tau}(\cB_\delta^\theta):\tau\in[\delta,1]\})=2, \qquad \dim(\{\phi_{-\tau}(\cB_1):\tau\in[0,1]\})=3,
\]
i.e., the intersection of a two- and a three-dimensional manifold in four-dimensional ambient space generically consists of one-dimensional curves. The slight shift of the boundary manifold from $\cB_0^\theta$ to $\cB_\delta^\theta$ avoids the problem that the vector field in~\eqref{eq:ODE_f} may, a priori, not be well-defined for $n\geq 1$ when $H=0$. Furthermore, we can give a more precise description of $\cB_\delta^\theta$ which is the core technical contribution for the proof of existence via a shooting argument. We have $\dim(\cB_\delta^\theta)=1$ and therefore, we may parametrize the solution manifold $\cB_\delta^\theta$ by the parameter $\kappa > 0$ if $\theta = 0$ and by the parameter $b \in \R$ if $\theta > 0$ (cf.~\S \ref{ssec:cw} and \S \ref{ssec:pw2}).
\begin{lemma}
\label{lem:poslittle}
Suppose $\left(x, H, H^\prime, H^{\prime\prime}\right) \in \cB_\delta^\theta$. Then it follows that $H>0$, $H^\prime > 0$, and furthermore:
\begin{itemize}
\item[(a)] if $\theta=0$ and $0 < n < 3/2$, we have $H' \le 0$ at some $x = x^* \le 1$ for $0 < \kappa \ll 1$ (\emph{undershoot}) and $H' > 0$ at $x = 1$ for $\kappa \gg 1$ (\emph{overshoot}).
\item[(b)] if $\theta>0$ then $H' \le 0$ at some $x = x^* \le 1$ for $- b \gg 1$ (\emph{undershoot}) and $H' > 0$ at $x = 1$ for $b \gg 1$ (\emph{overshoot}).
\end{itemize}
\end{lemma}
\begin{proof}[Proof of Lemma~\ref{lem:poslittle}, Part~(a)]
We first focus on the zero contact angle case $\theta = 0$. From \eqref{hdhd2h_0n32_1} (cf.~Cor.~\ref{cor:exp_0n32}) we necessarily have $H' > 0$ at $x = 1$ for $\kappa \gg 1$, that is, we have an overshoot for $\kappa \gg 1$. For proving the undershoot, we first assume $1 \le n \le 3/2$ and notice that
\begin{equation}\label{dh_int}
H' = H'(0) + H''(0) x + \int_0^x \int_0^{x_1} H'''(x_2) \, \d x_2 \, \d x_1 \stackrel{\eqref{eq:ODE_f}}{=} 2 \kappa x - \int_0^x \int_0^{x_1} (1-x_2) \left(H(x_1)\right)^{1-n} \, \d x_2 \, \d x_1.
\end{equation}
In order to estimate the double integral in \eqref{dh_int}, we observe
\begin{eqnarray*}
H &=& \kappa x^2 + \int_0^x \int_0^{x_1} \int_0^{x_2} H'''(x_3) \, \d x_3 \, \d x_2 \, \d x_1 \\
&\stackrel{\eqref{eq:ODE_f}}{=}& \kappa x^2 - \int_0^x \int_0^{x_1} \int_0^{x_2} (1-x_3) \left(H(x_3)\right)^{1-n} \, \d x_3 \, \d x_2 \, \d x_1 \le \kappa x^2
\end{eqnarray*}
for $0 \le x \le 1$. Inserted into \eqref{dh_int}, we obtain
\begin{eqnarray*}
H'(1) &\le& 2 \kappa - \kappa^{1-n} \int_0^1 \int_0^{x_1} \left(x_2^{2-2n} - x_2^{3-2n}\right) \, \d x_2 \, \d x_1 \\
&\le& 2 \kappa - \kappa^{1-n} \left(\left((4-2n) (3-2n)\right)^{-1} - \left((5-2n) (4-2n)\right)^{-1}\right) \\
&=& 2 \kappa \left(1 - \kappa^{-n} \left((5-2n) (4-2n) (3-2n)\right)^{-1}\right),
\end{eqnarray*}
proving the undershoot for $0 < \kappa \ll 1$ in this case.

\medskip

For $0 < n < 1$ we set $x_\kappa^* := \min \{x > 0: \, H'(x) = 0\} \cup \{1\}$ and argue similarly as in \eqref{dh_int} to obtain for $\delta \le x \le x^*_\kappa$
\begin{eqnarray}\nonumber
H' &\stackrel{\eqref{eq:ODE_f}}{=}& H'(\delta) + H''(\delta) (x-\delta) - \int_\delta^x \int_\delta^{x_1} (1-x_2) \left(H(x_2)\right)^{1-n} \, \d x_2 \, \d x_1 \\
&\le& H'(\delta) + H''(\delta) (x-\delta) - \left(H(\delta)\right)^{1-n} \int_\delta^x \int_\delta^{x_1} (1-x_2) \, \d x_2 \, \d x_1 \nonumber \\
&\le& H'(\delta) + H''(\delta) (x-\delta) - \frac 1 6 \left(H(\delta)\right)^{1-n} (x - \delta)^2 \left(3 - x + \delta\right). \label{dh_int_2}
\end{eqnarray}
Using \eqref{hdhd2h_0n32} (cf.~Cor.~\ref{cor:exp_0n32}), provided that $\delta \le x \le x^*_\kappa$, estimate~\eqref{dh_int_2} upgrades to
\begin{eqnarray*}
H' &\le& 2 \kappa \delta \left(1 + \cO\left(\kappa^{-n} \delta^{3-2n} + \kappa^{-n} \delta^{4-2n}\right)\right) + 2 \kappa \left(1 + \cO\left(\kappa^{-n} \delta^{3-2n} + \kappa^{-n} \delta^{4-2n}\right)\right) (x - \delta) \\
&& - \frac 1 6 \kappa^{1-n} \delta^{2-2n} \left(1 + \cO\left(\kappa^{-n} \delta^{3-2n} + \kappa^{-n} \delta^{4-2n}\right)\right) (x-\delta)^2 \left(3 - x + \delta\right).
\end{eqnarray*}
Taking $0 < \delta \ll \kappa^{n/(3-n)}$, we may conclude that there exist constants $c_1, c_2, c_3 > 0$ with
\begin{equation}\label{taylor_nl1}
H' \le c_1 \kappa^{1 + \frac{n}{3-2n}} + c_2 \kappa (x-\delta) - c_3 \kappa^{1-n + \frac{n (2-2n)}{3-2n}} (x-\delta)^2 \left(3 - x + \delta\right) \quad \mbox{for} \quad \delta \le x \le x_\kappa^*.
\end{equation}
Since the exponents of $\kappa$ in \eqref{taylor_nl1} fulfill
\[
1-n + \frac{n (2-2n)}{3-2n} < 1 < 1 + \frac{n}{3-2n},
\]
estimate~\eqref{taylor_nl1} yields the desired undershoot for $0 < \kappa \ll 1$.
\end{proof}
\begin{proof}[Proof of Lemma~\ref{lem:poslittle}, Part~(b)]
As in \eqref{dh_int_2}, we have for $0 < \delta \le x \le 1$
\[
H' \stackrel{\eqref{eq:ODE_f}}{=} H'(\delta) + H''(\delta) (x-\delta) - \int_0^x \int_0^{x_1} (1-x_2) \left(H(x_2)\right)^{1-n} \d x_2 \, \d x_1 \le H'(\delta) + H''(\delta) (x-\delta),
\]
since the double integral always yields a non-positive contribution. Using estimates~\eqref{as_hdhd2h_theta} of Corollary~\ref{cor:exp_hb}, we conclude that there exist constants $c_1, c_2 > 0$ with
\begin{equation}\label{under_tg0}
H' \le c_1 + c_2 b \left(x- \eps \verti{b}^{-1}\right) \quad \mbox{provided} \quad \eps \verti{b}^{-1} \le x \le 1, \quad 0 < \eps \ll 1, \quad \verti{b} \gg_\eps 1.
\end{equation}
Taking $b \to - \infty$, \eqref{under_tg0} yields the desired undershoot.

\medskip

Proving the overshoot is more cumbersome and we again distinguish between two cases:

\medskip

We start with the case $0 < n < 1$, restrict our considerations to $b > 0$, and estimate (using Corollary~\ref{cor:exp_hb})
\begin{eqnarray}\nonumber
H' &\stackrel{\eqref{eq:ODE_f}}{=}&  H'(\delta) + H''(\delta) \left(x - \delta\right) - \int_\delta^x \int_\delta^{x_1} (1-x_2) \left(H(x_2)\right)^{1-n} \, \d x_2 \, \d x_1 \\
&\stackrel{\eqref{as_hdhd2h_theta}}{\ge}& c_1 + c_2 b \left(x - \eps b^{-1}\right) - \int_{\eps b^{-1}}^x \int_{\eps b^{-1}}^{x_1} (1-x_2) \left(H(x_2)\right)^{1-n} \, \d x_2 \, \d x_1, \label{over_tg0_2}
\end{eqnarray}
with constants $c_1, c_2 > 0$ and where we assumed $x \ge \delta := \eps b^{-1}$, $0 < \eps \ll 1$, $b \gg_\eps 1$. In order to estimate the double integral in \eqref{over_tg0_2}, observe that for $\delta \le x \le 1$
\begin{eqnarray*}
H &\stackrel{\eqref{eq:ODE_f}}{=}& H(\delta) + H'(\delta) (x-\delta) + \frac 1 2 H''(\delta) (x-\delta)^2 \\
&&- \int_\delta^x \int_\delta^{x_1} \int_\delta^{x_2} (1-x_3) \left(H(x_3)\right)^{1-n} \, \d x_3 \, \d x_2 \, \d x_1\\
&\le& H(\delta) + H'(\delta) (x-\delta) + \frac 1 2 H''(\delta) (x-\delta)^2.
\end{eqnarray*}
Once more appealing to \eqref{as_hdhd2h_theta} (cf.~Cor.~\ref{cor:exp_hb}), we conclude that there exist constants $c_3, c_4, c_5 > 0$ with
\begin{equation}\label{over_tg0_1}
H \le c_3 \eps b^{-1} + c_4 \left(x- \eps b^{-1}\right) + c_5 b \left(x- \eps b^{-1}\right)^2 \quad \mbox{for} \quad \eps b^{-1} \le x \le 1,
\end{equation}
provided $0 < \eps \ll 1$ and $b \gg_\eps 1$. Hence, using \eqref{over_tg0_1} the double integral in \eqref{over_tg0_2} can be estimated as follows:
\begin{eqnarray}\nonumber
\lefteqn{- \int_\delta^x \int_\delta^{x_1} (1-x_2) \left(H(x_2)\right)^{1-n} \, \d x_2 \, \d x_1} \\
&\ge& - \int_{\eps b^{-1}}^x \int_{\eps b^{-1}}^{x_1} (1-x_2) \left(c_3 \eps b^{-1} + c_4 \left(x_2- \eps b^{-1}\right) + c_5 b \left(x_2- \eps b^{-1}\right)^2\right)^{1-n} \, \d x_2 \, \d x_1 \nonumber \\
&\ge& - \eps^{3-n} b^{n-3} \int_1^{b \eps^{-1} x} \int_1^{y_1}  \left(c_3 + c_4 \left(y_2- 1\right) + c_5 \eps \left(y_2- 1\right)^2\right)^{1-n} \, \d y_2 \, \d y_1 \nonumber \\
&=& \cO\left(b^{1-n} x^{4-2n}\right), \label{over_tg0_3}
\end{eqnarray}
where the dependence on the constant $\eps$ has been discarded in the last step. The combination of \eqref{over_tg0_2} and \eqref{over_tg0_3} shows
\[
H' \ge c_1 + c_2 b \left(x - \eps b^{-1}\right) + \cO\left(b^{1-n} x^{4-2n}\right), \quad \mbox{where} \quad \eps b^{-1} \le x \le 1, \quad 0 < \eps \ll 1, \quad b \gg_\eps 1,
\]
proving the overshoot as $b \to \infty$.

\medskip

Finally, for $1 \le n < 3$ we notice that for $0 < \eps \ll 1$ and $b \gg_\eps 1$ we have $H''\left(\eps b^{-1}\right) \sim b > 0$ by \eqref{as_hdhd2h_theta} (cf.~Cor.~\ref{cor:exp_hb}). Hence, $H'' > 0$ for $0 < x \ll 1$ in this case. Define $x^{**} := \min \{x > 0: \, H''(x) = 0\} \cup \{1\}$. Due to the boundary conditions \eqref{bcx1} necessarily $H \ge \theta x$ for $0 \le x \le x^{**}$. This allows to estimate for $\delta := \eps b^{-1} < x \le x^{**}$
\begin{eqnarray*}
H'' &\stackrel{\eqref{eq:ODE_f}}{=}& H''(\delta) - \int_\delta^x (1-x_1) \left(H(x_1)\right)^{1-n} \, \d x_1 \\
&\stackrel{\eqref{as_hdhd2h_theta}}{\ge}& c \, b - \theta^{1-n} \int_{\eps b^{-1}}^x (1-x_1) x_1^{1-n} \, \d x_1 \\
&=& c \, b - \theta^{1-n} \left(\frac{x^{2-n} - \eps^{2-n} b^{n-2}}{2-n} - \frac{x^{3-n} - \eps^{3-n} b^{n-3}}{3-n}\right), \label{over_tg0_4}
\end{eqnarray*}
with a constant $c > 0$, $0 < \eps \ll 1$ and $b \gg_\eps 1$. Estimate~\eqref{over_tg0_4} implies  an overshoot of $H''$ at $x = 1$ as $b \to \infty$ and as $H' = \theta > 0$ at $x = 0$, also an overshoot of $H'$ at $x = 1$ as $b \to \infty$.
\end{proof}

\begin{proposition}
\label{prop:exist}
For any $\delta>0$ sufficiently small, there exists a connecting orbit between $\cB_\delta^\theta$ and $\cB_1$. The orbit is unique if $H> 0$ for all $x\in(0,1]$.
\end{proposition}
\begin{proof}
The strategy is based upon a shooting argument. Fix some $\delta>0$ sufficiently small so that Lemma~\ref{lem:shift} and Lemma~\ref{lem:poslittle} hold. Consider an initial condition
\[
\left(\delta, H(\delta), H^\prime(\delta), H^{\prime\prime}(\delta)\right) \in \cB_\delta^\theta
\]
for an orbit $\gamma=\gamma(\tau)$. We start with the case $\theta=0$. Observe that we must require $H^{\prime\prime}(\delta)>0$ due to the local analytic expansion of the solution near $x=0$. However, $H^{\prime\prime}(\delta)$ can be taken as shooting parameter. As long as $H>0$ for $x > 0$ we have that $H^{\prime\prime}$ is monotonically decreasing due to the last line of \eqref{eq:ODE_f}. Since $H^\prime(\delta), H^{\prime\prime}(\delta) > 0$ any $\gamma$ starting in $\cB_\delta^\theta$ has $H$ and $H^\prime$ increasing initially due to the second and third line of \eqref{eq:ODE_f}. If $H^{\prime\prime}(\delta)>0$ is chosen sufficiently large, we can guarantee that the monotonically increasing results for $\gamma$ hold up to the time $\tau=1$ when it reaches $\{x=1\}$. In particular, we have an overshoot since the $H^\prime$-component of $\gamma(1)$ is positive; for the detailed estimates of this overshoot we refer to Lemma~\ref{lem:poslittle}(a) and its proof.  

Next, we observe that $H^\prime$ can only change sign from positive to negative after $H^{\prime\prime}$ has changed sign due to the last two lines of \eqref{eq:ODE_f}. Hence, one should select $H^{\prime\prime}(\delta)>0$ sufficiently small. By Lemma~\ref{lem:poslittle}(a) it follows that there exists a time $\tau_m\in(\delta,1)$ such that the $H^\prime$-component of $\gamma$ changes sign. This guarantees an undershoot of the terminal boundary condition $H^\prime = 0$ at $x = 1$. Notice that the undershoot is easy if $\kappa>0$ could be chosen freely as a small parameter but we have to guarantee that the local analytic series remains valid near the contact line, which is precisely why Lemma~\ref{lem:poslittle}(a)
is needed. 

Hence, if we vary the curvature $H^{\prime\prime}(\delta)>0$ from its initial large value, continuous dependence on initial conditions and the assumption $H>0$ for $x\in(0,1]$ implies that there exists a unique choice of $H^{\prime\prime}(\delta)$ such that $\gamma$ is a connecting orbit. The existence proof for $\theta>0$ carries over verbatim with the modification that we are not restricted to $H^{\prime\prime}(\delta)>0$ on $\cB_\delta^\theta$ and we have to apply Lemma~\ref{lem:poslittle}(b).
\end{proof}

\begin{proposition}
\label{prop:exist1}
For any $\delta>0$ sufficiently small, there exists a connecting orbit between $\cB_0^\theta$ and $\cB_1$. If $\theta=0$ and $H> 0$ for all $x\in(0,1]$, the orbit is unique and monotone, i.e., $H'>0$ for $x\in(0,1)$.
\end{proposition}
\begin{proof}
In the shooting argument we have shown the existence of a connecting orbit between $\cB_\delta^\theta$ and $\cB_1$ for any sufficiently small $\delta>0$, which guarantees that $\cB_\delta^\theta$ is a smooth real-analytic sub-manifold. In the proof of Proposition~\ref{prop:exist}, we have seen that the shooting parameter $H''(\delta)>0$ can be chosen within a compact subset $\cK_\delta\subset \cB_\delta^\theta$ to achieve an overshoot and an undershoot. Furthermore, the construction in Lemma~\ref{lem:poslittle} guarantees that the connecting orbit from $\cB_\delta^\theta$ to $\cB_1$ locally matches the expansion near the contact line. 

For $\theta=0$, the result about uniqueness follows from the existence proof. For monotonicity, suppose $H'$ vanishes for some $\tau\in(\delta,1)$, then $H''$ must have changed sign previously as the solution starts with positive curvature near the contact line. Hence, the signs of the vector field~\eqref{eq:ODE_f} yield that $H'(1)=0$ cannot occur as $H>0$ on $(0,1]$ implies that $H''$ is decreasing on $(0,1]$.
\end{proof}

It should be noted that it is unclear what happens if we allow for the possibility that $H(x_*)=0$ for some $x_*\in(0,1)$ in the arguments above. In this case, the term $H^{n-1}$ does not directly make sense if $H'(x_*)<0$ since the second line of \eqref{eq:ODE_f} implies that $H$ would become negative. However, the next direct argument shows that for nonzero contact angles the solution must be unique; we remark that the uniqueness result for zero contact angles is well-known (cf.~\cite[Th.~1.2]{bpw.1992}).

\subsection{Uniqueness for nonzero dynamic contact angles\label{ssec:global_unique}}
The proof of uniqueness of classical solutions to~\eqref{problem_x} with $\theta > 0$ closely follows the argumentation in \cite[\S 6]{bpw.1992}. For $n = 1$, \eqref{problem_x} reduces to a third-order ODE with three boundary conditions, thus directly leading to the unique explicit representation~\eqref{n1_explicit}. We treat the cases $n \in (0,1)$ and $n \in (1,3)$ separately.
\begin{proposition}
Suppose $\theta > 0$ and $n\in(1,3)$. Then the solution to~\eqref{problem_x} is unique.
\end{proposition}
\begin{proof}
Suppose that there are two solutions $H_1$ and $H_2$ to~\eqref{problem_x}. We define the auxiliary function
\begin{equation}\label{def_phi}
\Phi := \left(H_1 - H_2\right) \frac{\d^2}{\d x^2} \left(H_1 - H_2\right) - \frac 1 2 \left(\frac{\d}{\d x} \left(H_1 - H_2\right)\right)^2.
\end{equation}
We first show that in fact $\Phi = 0$ at $x = 0$. Therefore note that $H_1 - H_2 = o(x)$, $\frac{\d}{\d x} (H_1 - H_2) = o(1)$ as $x \searrow 0$ by~\eqref{bcx1}, and that for $\eps > 0$
\begin{eqnarray*}
\frac{\d^2}{\d x^2} \left(H_1 - H_2\right) &\stackrel{\eqref{ode_third2}}{=}& - \int_x^\eps \left(-1+x^\prime\right) \left(H_1^{1-n}\left(x^\prime\right) - H_2^{1-n}\left(x^\prime\right)\right) \d x^\prime + C(\eps) \\
&\stackrel{\eqref{bcx1}}{=}& - \int_x^\eps \left(-1+x^\prime\right) o\left(\left(x^\prime\right)^{1-n}\right) \d x^\prime + C(\eps) \\
&=& o\left(x^{2-n}\right) + \tilde C(\eps),
\end{eqnarray*}
where $C(\eps)$ and $\tilde C(\eps)$ are $\eps$-dependent constants. Using~\eqref{def_phi}, this yields 
\begin{equation}\label{bc_phi}
\Phi = \underbrace{\left(H_1 - H_2\right)}_{= o(x)} \underbrace{\frac{\d^2}{\d x^2} \left(H_1 - H_2\right)}_{= o\left(x^{2-n}\right) + \tilde C(\eps)} - \frac 1 2 \underbrace{\left(\frac{\d}{\d x} \left(H_1 - H_2\right)\right)^2}_{= o(1)} = o(1) \quad \mbox{as} \quad x \searrow 0.
\end{equation}
Then we may infer again using~\eqref{def_phi}
\[
\frac{\d \Phi}{\d x} = \left(H_1 - H_2\right) \frac{\d^3}{\d x^3} \left(H_1 - H_2\right) \stackrel{\eqref{ode_third2}}{=} \left(H_1 - H_2\right) \left(-1+x\right) \left(H_1^{1-n} - H_2^{1-n}\right),
\]
that is, $\frac{\d \Phi}{\d x} \ge 0$ in $(0,1)$ because of $n > 1$. Due to the boundary condition \eqref{bc_phi} this yields $\Phi \ge 0$ in $(0,1)$, which by~\eqref{def_phi} implies
\[
\left(H_1 - H_2\right) \frac{\d^2}{\d x^2} \left(H_1 - H_2\right) \ge \frac 1 2 \left(\frac{\d}{\d x} \left(H_1 - H_2\right)\right)^2 \ge 0.
\]
The latter leads to
\begin{align*}
\frac{\d^2}{\d x^2} \left(H_1 - H_2\right)^2 &= 2 \frac{\d}{\d x} \left(\left(H_1 - H_2\right) \frac{\d}{\d x} \left(H_1 - H_2\right)\right) \\
&= 2 \left(H_1 - H_2\right) \frac{\d^2}{\d x^2} \left(H_1 - H_2\right) + 2 \left(\frac{\d}{\d x} \left(H_1 - H_2\right)\right)^2 \\
&\ge 3 \left(\frac{\d}{\d x} \left(H_1 - H_2\right)\right)^2 \ge 0.
\end{align*}
Hence, $\frac{\d}{\d x} \left(H_1 - H_2\right)^2$ is an increasing function in $(0,1)$ with boundary values
\[
\frac{\d}{\d x} \left(H_1 - H_2\right)^2 \stackrel{\eqref{bcx1}}{=} 0 \quad \mbox{at} \quad x = 0 \qquad \mbox{and} \qquad \frac{\d}{\d x} \left(H_1 - H_2\right)^2 \stackrel{\eqref{bcx2}}{=} 0 \qquad \mbox{at} \quad x = 1.
\]
Necessarily we have $\frac{\d}{\d x} \left(H_1 - H_2\right)^2 \equiv 0$ in $(0,1)$ and once more appealing to~\eqref{bcx1} we conclude $H_1 \equiv H_2$ in $(0,1)$.
\end{proof}
\begin{proposition}
Suppose $\theta > 0$ and $n\in(0,1)$. Then the solution to~\eqref{problem_x} is unique.
\end{proposition}
\begin{proof}
For $n \in (0,1)$, the auxiliary function $\Phi$ from~\eqref{def_phi} is not necessarily monotonic and therefore a different approach has to be used. Here, another scaling transformation than the one used to derive~\eqref{problem_x} turns out to be convenient, that is, we may define $\tilde H := H/H(1)$, $a:= \left(H(1)\right)^{-n}$, $\tilde{\theta} := (H(1))^{n-1} \theta$, and $\tilde x := a (1-x)$, so that~\eqref{problem_x} is equivalent to
\begin{subequations}\label{problem_til}
\begin{align}
\tilde H^{n-1} \frac{\d^3 \tilde H}{\d x^3} &= \tilde x \quad \mbox{for} \quad \tilde x \in (0,a),\label{eq_til}\\
\left(\tilde H, \frac{\d \tilde H}{\d \tilde x}\right) &= \left(0,-\tilde{\theta}\right) \quad \mbox{at} \quad \tilde x = a,\\
\left(\tilde H, \frac{\d \tilde H}{\d \tilde x}\right) &= (1,0) \quad \mbox{at} \quad \tilde x = 0,\label{bc_til2}
\end{align}
\end{subequations}
where $\tilde H$ and $a$ are the unknowns. Now suppose that we have two solutions $\left(H_1,a_1\right)$ and $\left(H_2,a_2\right)$ of~\eqref{problem_til}. Defining $\Psi := \tilde H_1 - \tilde H_2$, we have $\Psi = \frac{\d \Psi}{\d \tilde x} = 0$ at $\tilde x = 0$ due to \eqref{bc_til2}. As equation~\eqref{eq_til} is non-degenerate in $\tilde x = 0$, by standard theory necessarily $\frac{\d^2 \Psi}{\d \tilde x^2} \ne 0$ at $\tilde x = 0$ and we may assume without loss of generality $\frac{\d^2 \Psi}{\d \tilde x^2} > 0$ at $\tilde x = 0$. Because of $\frac{\d^3 \Psi}{\d x^3} \stackrel{\eqref{eq_til}}{=} \tilde x \left(\tilde H_1^{1-n} - \tilde H_2^{1-n}\right)$ in $\left(0,\min\{a_1,a_2\}\right)$, this implies $a_1 > a_2$. Furthermore, $\Psi$ has to be strictly increasing on the interval $(0,a_2)$, which implies
\begin{equation}\label{ineq_h1}
\frac{\d \tilde H_1}{\d \tilde x}(a_2) = \frac{\d \Psi}{\d x}(a_2) > \frac{\d \Psi}{\d x}(0) = 0.
\end{equation}
On the other hand, we know that $\frac{\d \tilde H_1}{\d \tilde x} = 0$ at $\tilde x = 0$, $\frac{\d \tilde H_1}{\d \tilde x} < 0$ at $\tilde x = a_1$, and that $\frac{\d \tilde H_1}{\d \tilde x}$ is a strictly convex function in $(0,a_1)$ by~\eqref{eq_til}. Therefore, $\frac{\d \tilde H_1}{\d \tilde x} < 0$ in $(0,a_1)$, which contradicts \eqref{ineq_h1}.
\end{proof}

\appendix

\section{Higher Dimensions}\label{app:higher}
In this appendix, we discuss the regularity of source-type self-similar solutions in arbitrary dimensions $d \in \N$. Then the thin-film equation \eqref{tfe} reads
\begin{subequations}\label{source_pr_high}
\begin{equation}\label{tfe_higher}
\partial_t h + \nabla \cdot \left(h^n \nabla \Delta h\right) = 0 \quad \mbox{for } \, (t,z) \in \{h > 0\}.
\end{equation}
The initial condition is
\begin{equation}\label{init_high}
\lim_{t \searrow 0} h = M \delta_0 \quad \mbox{in} \quad \mathcal D^\prime\left(\R^d\right),
\end{equation}
\end{subequations}
where $M > 0$ denotes the mass of the fluid film. Existence, uniqueness, and leading-order asymptotics of this problem were studied by Bernis and Ferreira in \cite{bf.1997} under the assumption
\begin{equation}\label{similar_high}
h(t,z) = t^{-\frac{d}{n d + 4}} H(Z) \quad \mbox{with } \, Z = t^{-\frac{1}{n d + 4}} \verti{z},
\end{equation}
where $H \in C^1(\R) \cap C^3\left(\{H > 0\}\right) \cap L^1(\R)$ and $H^n \frac{\d^3 H}{\d Z^3} \in C^1\left(\{H > 0\}\right)$ (corresponding to the zero static contact-angle case). Then one may verify
\begin{align*}
\partial_t h &= - \frac{1}{n d + 4} t^{-\frac{(n+1) d + 4}{n d + 4}} Z^{1-d} \frac{\d}{\d Z} \left(Z^d H\right), \\
\nabla \cdot \left(h^n \nabla \Delta h\right) &= t^{- \frac{(n+1) d + 4}{n d + 4}} Z^{1-d} \frac{\d}{\d Z} \left(Z^{d-1} H^n \frac{\d}{\d Z} \left(Z^{1-d} \frac{\d}{\d Z} \left(Z^{d-1} \frac{\d H}{\d Z}\right)\right)\right),
\end{align*}
and therefore
\begin{equation}\label{four_high}
\frac{\d}{\d Z} \left(Z^{d-1} H^n \frac{\d}{\d Z} \left(Z^{1-d} \frac{\d}{\d Z} \left(Z^{d-1} \frac{\d H}{\d Z}\right)\right)\right) = \frac{1}{n d + 4} \frac{\d}{\d Z} \left(Z^d H\right) \quad \mbox{in } \, \{H > 0\}.
\end{equation}
We may assume the boundary condition $H = \frac{\d H}{\d Z} = 0$ at $Z \in \partial \{H > 0\}$, so that one trivial integration yields
\begin{equation}\label{three_high}
H^{n-1} \frac{\d}{\d Z} \left(Z^{1-d} \frac{\d}{\d Z} \left(Z^{d-1} \frac{\d H}{\d Z}\right)\right) = \frac{1}{n d + 4} Z \quad \mbox{in } \, \{H > 0\}.
\end{equation}
By the same reasoning as in Section~\ref{sec:transformations}, we may assume that the free boundary is at the positions $\pm Z_0$ with some constant $Z_0 > 0$. By rescaling without loss of generality $Z_0 = 1$ and by setting $x := Z+1$, we are led to study the problem (compare to \eqref{problem_x})
\begin{subequations}\label{pr_x_high}
\begin{align}
H^{n-1} \left(\frac{\d^3 H}{\d x^3} - \frac{d-1}{1-x} \frac{\d^2 H}{\d x^2} - \frac{d-1}{(1-x)^2} \frac{\d H}{\d x}\right) &= -1+x\quad \mbox{for } \, x \in \left(0,1\right),\label{ode_third_high2}\\
\left(H,\frac{\d H}{\d x}\right) &= \left(0,0\right) \quad \mbox{at } \, x = 0,\label{bcx_high1}\\
\frac{\d H}{\d x} &= 0 \quad \mbox{at } \, x = 1,\label{bcx_high2}
\end{align}
\end{subequations}
with $H \in C^3((0,1)) \cap C^1([0,1])$ (\emph{classical solutions}). The result in \cite{bf.1997} implies existence and uniqueness of classical solutions to \eqref{pr_x_high}. Unlike in the $1+1$-dimensional case, where we have also constructed solutions in the case of nonzero dynamic contact angles, here we only consider zero static contact angles. It seems that our method of constructing solutions does not apply in an apparent way to \eqref{pr_x_high} as singular terms of \eqref{ode_third_high2} at the symmetry point $x = 1$ appear. In some sense, these terms slow down the dynamics and move the symmetry point to infinity in de-singularized coordinates, which makes the construction of solutions by a shooting argument, starting at the contact line, less accessible and a shooting argument as in \cite{bf.1997}, starting at the point of symmetry, more favorable.

\medskip

For \eqref{pr_x_high} we can show the following results:
\begin{theorem}\label{th:high_0n32_0}
Suppose $n \in (0,3/2)$ and $H>0$ for $x\in(0,1]$. Then the unique solution $H$ of problem~\eqref{problem_x} fulfills the asymptotic
\[
H = \kappa x^2 \left(1 + v\left(x,x^\beta\right)\right) \quad \mbox{for } \, 0 \le x \ll 1,
\]
where $\kappa > 0$ is an $n$-dependent constant, $\beta = 3 - 2 n$, and $v = v(x_1,x_2)$ is an analytic function in a neighborhood of $(x_1,x_2) = (0,0)$, $v(0,0)=0$ but $v\not\equiv 0$ with a non-trivial dependence on both $x_1$ and $x_2$.
\end{theorem}
\begin{theorem}\label{th:high_32n3}
Suppose $n \in (3/2,3)$ and $H>0$ for $x\in(0,1]$. Then the unique solution $H$ of problem~\eqref{problem_x} fulfills the asymptotic
\[
H = \mu^{- \frac 1 n} x^{\frac 3 n} \left(1 + v\left(x,x^\beta\right)\right) \quad \mbox{for } \, 0 \le x \ll 1,
\]
where $\mu = \frac 3 n \left(\frac 3 n - 1\right) \left(2 - \frac 3 n\right)$, $\beta = \frac{\sqrt{-27+36n-8n^2}-9+4n}{2n} \in (0,1)$, and $v = v(x_1,x_2)$ is an analytic function in a neighborhood of $(x_1,x_2) = (0,0)$, $v(0,0)=0$ but $v\not\equiv 0$ with an in general non-trivial dependence on both $x_1$ and $x_2$.
\end{theorem}
Theorem~\ref{th:high_0n32_0} is the generalization of Theorem~\ref{th:0n32_0} to higher dimensions, whereas Theorem~\ref{th:high_32n3} generalizes the one-dimensional result \cite[Th.~1]{ggo.2013}. The leading-order expansions
\[
H = \begin{cases} \kappa x^2 (1 + o(1)) \quad \mbox{as } \, x \searrow 0 & \mbox{for } \, 0 < n < 3/2 \\ \mu^{\frac 1 n} x^{\frac 3 n} (1 + o(1)) \quad \mbox{as } \, x \searrow 0 & \mbox{for } \, 3/2 < n < 3 \end{cases}
\]
were already proven in \cite[Th.~1.3]{bf.1997}.

\begin{proof}[Proof of Theorem~\ref{th:high_0n32_0}]
The proof follows the lines of the proof of Theorem~\ref{th:0n32_0}, which is why we keep the presentation brief and only point out some essential differences. Following the reasoning in Section~\ref{ssec:cw}, we may set $H =: x^2 F$ for $0 \le x \ll 1$, so that -- after passing to $s := \log x$ (cf.~\eqref{log_coord}) -- \eqref{ode_third_high2} transforms into
\begin{equation}\label{ode_high_t1}
\begin{aligned}
& \frac{\d^3 F}{\d s^3} + 3 \frac{\d^2 F}{\d s^2} + 2 \frac{\d F}{\d s} - \frac{(d-1) \txte^s}{1 - \txte^s} \left(\frac{\d^2 F}{\d s^2} + 3 \frac{\d F}{\d s} + 2 F\right) - \frac{(d-1) \txte^{2s}}{(1-\txte^s)^2} \left(\frac{\d F}{\d s} + 2 F\right) \\
& \quad = \frac{-\txte^{(3-2n)s} + \txte^{(4-2n)s}}{F^{n-1}}.
\end{aligned}
\end{equation}
Compared to \eqref{not_0n32} we define more generally
\begin{equation}\label{not_0n32_high}
x_1 := \txte^s, \quad x_2 := \txte^{(3-2n) s}, \quad F^\prime := \frac{\d F}{\d s}, \quad \mbox{and } \, F^{\prime\prime} := \frac{\d^2 F}{\d s^2},
\end{equation}
so that \eqref{ode_high_t1} can be recast as a five-dimensional autonomous dynamical system
\begin{equation}\label{dyn_0n32_high}
\begin{aligned}
\frac{\d}{\d s} \begin{pmatrix} x_1 \\ x_2 \\ F \\ F^\prime \\ F^{\prime\prime} \end{pmatrix} = \mathcal F\left(x_1,x_2,F,F^\prime,F^{\prime\prime}\right) := \, & \begin{pmatrix} x_1 \\ (3-2n) x_2 \\ F^\prime \\ F^{\prime\prime} \\ \mathcal F_5\left(x_1,x_2,F,F^\prime,F^{\prime\prime}\right) \end{pmatrix} \quad \mbox{for } \, - \infty < s \ll 1, \\
\mbox{with} \quad \mathcal F_5\left(x_1,x_2,F,F^\prime,F^{\prime\prime}\right) := \, & \frac{x_1 x_2 - x_2}{F^{n-1}} - 3 F^{\prime\prime} - 2 F^\prime \\
& + \frac{(d-1) x_1}{1-x_1} \left(F^{\prime\prime} + 3 F^\prime + 2 F\right) \\
& + \frac{(d-1) x_1^2}{(1-x_1)^2} \left(F^\prime + 2 F\right).
\end{aligned}
\end{equation}
For $d \ge 2$ it appears to be unavoidable to choose the coordinates as in \eqref{not_0n32_high} and not as in \eqref{not_0n32}, as the additional terms in the last two lines of \eqref{dyn_0n32_high} can otherwise not be written in an analytic form at the stationary point $\left(x_1, x_2, F, F^\prime, F^{\prime\prime}\right) = (0,0,\kappa,0,0) =: p_\kappa$ with $\kappa > 0$ to which the solution converges as $s \to - \infty$ (cf.~\cite[Th.~1.3]{bf.1997} and the analogue of Lemma~\ref{lem:conv_0n32} for higher dimensions). As the linearization of $\mathcal F$ in $p_\kappa$ is given by
\[
\txtD \mathcal F (p_\kappa) = \begin{pmatrix} 1 & 0 & 0 & 0 & 0\\ 0 & 3-2n & 0 & 0 & 0\\ 0 & 0 & 0 & 1 & 0\\ 0 & 0 & 0 & 0 & 1\\ 2 (d-1) \kappa & - \kappa^{1-n} & 0 & -2 & -3 \end{pmatrix},
\]
its characteristic polynomial can be calculated to be
\[
\zeta \mapsto (\zeta-1) (\zeta-(3-2n)) \underbrace{\zeta (\zeta+1) (\zeta+2)}_{\stackrel{\eqref{def_q}}{=} q(\zeta)}
\]
and the same reasoning as in the proofs of Proposition~\ref{prop:steady_0n32} and Lemma~\ref{lem:ntcw} applies.
\end{proof}

\begin{proof}[Proof sketch of Theorem~\ref{th:high_32n3}]
Here, we closely follow the arguments of \cite[\S 2.2]{ggo.2013} and \cite[\S 2]{ggo.2015}. First, we factor off the leading-order behavior (cf.~\cite[Th.~1.3]{bf.1997}) by setting
\[
H =: \mu^{- \frac 1 n} x^{\frac 3 n} F \quad \mbox{with } \, \mu = \frac 3 n \left(\frac 3 n - 1\right) \left(2 - \frac 3 n\right),
\]
so that \eqref{ode_third_high2} can be recast in the variable $s := \log x$ (cf.~\eqref{log_coord}) as
\begin{equation}\label{ode_high_t2}
\begin{aligned}
& \frac{\d^3 F}{\d s^3} + 3 \left(\frac 3 n - 1\right) \frac{\d^2 F}{\d s^2} + \left(3 \left(\frac 3 n\right)^2 - 6 \frac 3 n + 2\right) \frac{\d F}{\d s} - \frac 3 n \left(\frac 3 n - 1\right) \left(2 - \frac 3 n\right) F\\
& - \frac{(d-1) \txte^s}{1-\txte^s} \left(\frac{\d^2 F}{\d s^2} + \left(2 \frac 3 n - 1\right) \frac{\d F}{\d s} + \frac 3 n \left(\frac 3 n - 1\right) F\right) \\
& - \frac{(d-1) \txte^{2s}}{(1-\txte^s)^2} \left(\frac{\d F}{\d s} + \frac 3 n F\right) \\
& \quad = \mu \frac{-1+\txte^s}{F^{n-1}}
\end{aligned}
\end{equation}
Now we may use the independent variables
\[
x = \txte^s, \quad F, \quad F^\prime := \frac{\d F}{\d s}, \quad \mbox{and } \, F^{\prime\prime} := \frac{\d^2 F}{\d s^2},
\]
for which \eqref{ode_high_t2} can be reformulated as a four-dimensional autonomous dynamical system of the form
\begin{equation}\label{dynsys_32n3}
\begin{aligned}
\frac{\d}{\d s} \begin{pmatrix} x \\ F \\ F^\prime \\ F^{\prime\prime} \end{pmatrix} = \mathcal F\left(x,F,F^\prime,F^{\prime\prime}\right) := \, & \begin{pmatrix} x \\ F^\prime \\ F^{\prime\prime} \\ \mathcal F_4\left(x,F,F^\prime,F^{\prime\prime}\right) \end{pmatrix} \quad \mbox{for } \, - \infty < s \ll 1, \\
\mbox{with} \quad \mathcal F_4\left(x,F,F^\prime,F^{\prime\prime}\right) := \, & \mu \frac{-1+x}{F^{n-1}} \\
&- 3 \left(\frac 3 n - 1\right) F^{\prime\prime} - \left(3 \left(\frac 3 n\right)^2 - 6 \frac 3 n + 2\right) F^\prime + \mu F \\
& + \frac{(d-1) x}{1-x} \left(F^{\prime\prime} + \left(2 \frac 3 n - 1\right) F^\prime + \frac 3 n \left(\frac 3 n - 1\right) F\right) \\
& + \frac{(d-1) x^2}{(1-x)^2} \left(F^\prime + \frac 3 n F\right),
\end{aligned}
\end{equation}
where the flow is apparently analytic in a neighborhood of the stationary point $\left(x,F,F^\prime,F^{\prime\prime}\right) = (0,1,0,0) =: p$, to which the solution converges as $s \to - \infty$ (the proof follows the lines of the proof of Lemma~\ref{lem:conv_0n32}). The linearization of the flow in $p$ is given by 
\[
\txtD \mathcal F(p) = \begin{pmatrix} 1 & 0 & 0 & 0 \\ 0 & 0 & 1 & 0 \\ 0 & 0 & 0 & 1 \\ \mu + (d-1) \frac 3 n \left(\frac 3 n - 1\right) & n \mu & - 3 \left(\frac 3 n\right)^2 + 6 \frac 3 n - 2 & - 3 \left(\frac 3 n - 1\right) \end{pmatrix}
\]
and the characteristic polynomial follows as
\[
\zeta \mapsto (\zeta-1) (\zeta - \beta) (\zeta+1) (\zeta + \alpha)
\]
with roots
\[
\alpha := \frac{- \sqrt{-27+36n-8n^2}-9+4n}{2n} \in (-2,0), \quad \beta := \frac{\sqrt{-27+36n-8n^2}-9+4n}{2n} \in (0,1).
\]
In particular the flow $\mathcal F$ \emph{is} hyperbolic in $p$. Now, we may determine the eigenspace $E^\txtu(p)$ of the positive eigenvalues $1$ and $\beta$ (the tangent space to $W^\txtu(p)$), which is spanned by the eigenvectors
\[
v_1 := \left(\frac{2 (3-n) n}{(d+1) n - 3}, 1, 1, 1\right)^\top \qquad \mbox{and} \qquad v_2 := \left(0, \frac{1}{\beta^2}, \frac 1 \beta, 1\right)^\top,
\]
and therefore given by
\begin{subequations}\label{eigen_32n3}
\begin{align}
\frac{(1-\beta) ((d+1)n-3)}{2 (3-n) n} x + \beta F - F^\prime &= \beta, \label{eigen_32n3_1}\\
\frac{(1-\beta^2) ((d+1)n-3)}{2 (3-n) n} x + \beta^2 F - F^{\prime\prime} &= \beta^2. \label{eigen_32n3_2}
\end{align}
\end{subequations}
In particular, we can write $F^\prime$ as a graph of $x$ and $F$ on $W^\txtu(p)$, that is, we can write
\begin{equation}\label{graph_32n3}
\frac{\d u}{\d s} = G(x,u) \quad \mbox{with } \, u := F-1,
\end{equation}
where $G = G(x,u)$ is an analytic function in a neighborhood of $(x,u) = (0,0)$ with $G(0,0) = 0$. Additionally knowing that $\frac{\partial G}{\partial u}(0,0) = \beta$ (this can be read off from \eqref{eigen_32n3_1}), we can recast \eqref{graph_32n3} as
\begin{equation}\label{graph_32n3_2}
\left(\frac{\d}{\d s} - \beta\right) u = G(x,u) - \beta u.
\end{equation}
Then one may replace the operator $\frac{\d}{\d s} = x \frac{\d}{\d x}$ by $x_1 \partial_{x_1} + \beta x_2 \partial_{x_2}$ and $u = u(x)$ by $v = v\left(x_1,x_2\right)$ and solve instead of \eqref{graph_32n3_2} the unfolded equation
\[
\left(x_1 \partial_{x_2} + \beta x_2 \partial_{x_2} - \beta\right) v = G\left(x_1,v\right) - \beta v,
\]
leading to the fixed-point problem
\begin{equation}\label{32n3_unfold}
v(x_1,x_2) = - b x_2 + \int_0^1 r^{-\beta} \left(G\left(r x_1, v\left(r x_1 r^\beta x_2\right)\right) - \beta v\left(r x_1 r^\beta x_2\right)\right) \frac{\d r}{r},
\end{equation}
where $b > 0$ is a real parameter. This fixed-point problem can be solved by applying the contraction-mapping theorem. We refer to Section~\ref{ssec:pw2} or \cite[\S 3.,4]{ggo.2013} for more details in analogous cases.
\end{proof}

%
\bibliography{tfe_dynsys} 
\bibliographystyle{plain} 

\end{document}